\documentclass[11pt,a4paper,reqno]{amsart}
\usepackage[english]{babel}
\usepackage[T1]{fontenc}
\usepackage{verbatim}
\usepackage{palatino}
\usepackage{amsmath}
\usepackage{mathabx}
\usepackage{amssymb}
\usepackage{amsthm}
\usepackage{amsfonts}
\usepackage{graphicx}
\usepackage{esint}
\usepackage{color}
\usepackage{mathtools}
\usepackage{overpic}

\usepackage[colorlinks = true, citecolor = black]{hyperref}
\author{Tuomas Orponen, Carmelo Puliatti, and Aleksi Py\"or\"al\"a}
\title[Fourier transforms of fractal measures]{On Fourier transforms of fractal measures \\ on the parabola}
\address{Department of Mathematics and Statistics\\ University of Jyv\"askyl\"a,
P.O. Box 35 (MaD)\\
FI-40014 University of Jyv\"askyl\"a\\
Finland} \email{tuomas.t.orponen@jyu.fi}
\email{aleksi.v.pyorala@jyu.fi}

\address{Department of Mathematics and Statistics\\ University of Jyv\"askyl\"a,
	P.O. Box 35 (MaD)\\
	FI-40014 University of Jyv\"askyl\"a\\
	Finland} 
\address{Current address:
Departament de Matem\`atiques\\ Universitat Aut\`onoma de Barcelona, 08193 Bellaterra, Barcelona, Catalonia}
\email{carmelo.puliatti@uab.cat}

\date{\today}
\thanks{T.O. is supported by the Research Council of Finland via the project \emph{Approximate incidence geometry}, grant no. 355453, and by the European Research Council (ERC) under the European Union’s Horizon Europe research and innovation programme (grant agreement No 101087499). C.P. is supported by the Research Council of Finland via the project \emph{Singular integrals, harmonic functions, and boundary regularity in Heisenberg groups}, grant no. 352649. A.P. is supported by the Research Council of Finland via the project \emph{GeoQuantAM: Geometric and Quantitative Analysis on Metric spaces}, grant no. 354241.}

\newcommand{\R}{\mathbb{R}}

\newcommand{\N}{\mathbb{N}}

\newcommand{\Z}{\mathbb{Z}}

\newcommand{\spt}{\operatorname{spt}}
\newcommand{\Hd}{\dim_{\mathrm{H}}}

\newcommand{\diam}{\operatorname{diam}}

\newcommand{\dist}{\operatorname{dist}}

\def\Barint_#1{\mathchoice
          {\mathop{\vrule width 6pt height 3 pt depth -2.5pt
                  \kern -8pt \intop}\nolimits_{#1}}%
          {\mathop{\vrule width 5pt height 3 pt depth -2.6pt
                  \kern -6pt \intop}\nolimits_{#1}}%
          {\mathop{\vrule width 5pt height 3 pt depth -2.6pt
                  \kern -6pt \intop}\nolimits_{#1}}%
          {\mathop{\vrule width 5pt height 3 pt depth -2.6pt
                  \kern -6pt \intop}\nolimits_{#1}}}

\numberwithin{equation}{section}

\theoremstyle{plain}
\newtheorem{thm}[equation]{Theorem}
\newtheorem*{"thm"}{"Theorem"}

\newtheorem{lemma}[equation]{Lemma}

\newtheorem{cor}[equation]{Corollary}
\newtheorem{proposition}[equation]{Proposition}
\newtheorem{question}{Question}
\newtheorem{claim}[equation]{Claim}

\theoremstyle{definition}

\newtheorem{definition}[equation]{Definition}

\newtheorem{notation}[equation]{Notation}

\theoremstyle{remark}
\newtheorem{remark}[equation]{Remark}

\newcommand{\nref}[1]{(\hyperref[#1]{#1})}

\DeclareMathSymbol{\intop}  {\mathop}{mathx}{"B3}

\begin{document}

\begin{abstract} Let $s \in [0,1]$ and $t \in [0,\min\{3s,s + 1\})$. Let $\sigma$ be a Borel measure supported on the parabola $\mathbb{P} = \{(x,x^{2}) : x \in [-1,1]\}$ satisfying the $s$-dimensional Frostman condition $\sigma(B(x,r)) \leq r^{s}$. Answering a question of the first author, we show that there exists an exponent $p = p(s,t) \geq 1$ such that
$$\|\hat{\sigma}\|_{L^{p}(B(R))} \leq C_{s,t}R^{(2 - t)/p}, \qquad R \geq 1.$$
Moreover, when $s \geq 2/3$ and $t \in [0,s + 1)$, the previous inequality is true for $p \geq 6$.

We also obtain the following fractal geometric counterpart of the previous results. If $K \subset \mathbb{P}$ is a Borel set with $\Hd K = s \in [0,1]$, and $n \geq 1$ is an integer, then 
$$ \dim_{\mathrm{H}}(nK) \geq \min\{3s - s \cdot 2^{-(n - 2)},s + 1\}.$$
 \end{abstract}

\maketitle

\tableofcontents

\section{Introduction}

The purpose of this paper is to prove the following theorem:

\begin{thm}\label{main} For every $0 \leq s \leq 1$ and $t \in [0,\min\{3s,s + 1\})$, there exists $p = p(s,t) \geq 1$ such that the following holds. Let $\sigma$ be a Borel measure on $\mathbb{P} = \{(x,x^{2}) : x \in [-1,1]\}$ satisfying $\sigma(B(x,r)) \leq r^{s}$ for all $x \in \R^{2}$ and $r > 0$. Then, 
\begin{equation}\label{form83} \|\hat{\sigma}\|_{L^{p}(B(R))} \leq C_{s,t}R^{(2 - t)/p}, \qquad R \geq 1. \end{equation}
When $s \geq 2/3$, one may take $p(s,t) = 6$. \end{thm}
This verifies \cite[Conjecture 1.6]{MR4528124}, and is also related to \cite[Theorem 1.1]{MR4661079}, see Section \ref{s:proofIdeas} for more details. The exponent $\min\{3s,s + 1\}$ is optimal. The detailed argument is given in \cite[Example 1.8]{MR4528124}, but here is the idea. Assume that $\sigma$ is a probability measure supported in the $\delta$-neighbourhood of
\begin{displaymath} \{(x,x^{2}) : x \in (\delta^{s}\mathbb{Z}) \cap [-1,1]\} \subset \mathbb{P}. \end{displaymath}
Such a measure $\sigma$ can satisfy $\sigma(B(x,r)) \lesssim r^{s}$. Then $\sigma$ is supported in the $\delta$-neighbourhood of $(\delta^{s} \mathbb{Z} \times \delta^{2s}\mathbb{Z}) \cap [-1,1]^{2}$. If $s < \tfrac{1}{2}$, this product looks like a $3s$-dimensional arithmetic progression. If $s \geq \tfrac{1}{2}$, the product looks like the $\delta$-neighbourhood of $(\delta^{s}\Z \times \R) \cap [-1,1]^{2}$. In particular, the support of $\sigma$, and also the supports of arbitrarily high convolution powers $\sigma \ast \cdots \ast \sigma$, are "trapped" $\delta$-close to an arithmetic progression of dimension $\min\{3s,s + 1\}$.

 The following fractal geometric result is very closely connected to Theorem \ref{main}:

\begin{thm}\label{main2} Assume that $K \subset \mathbb{P}$ is a Borel set with $\Hd K \geq s$. Then, the Hausdorff dimension of the $n$-fold sum-set $nK = K + \ldots + K$ satisfies
\begin{displaymath} \Hd (nK) \geq \min\{3s - s \cdot 2^{-(n - 2)},s + 1\}, \qquad n \geq 1. \end{displaymath}
In particular, $\Hd (K + K + K) \geq \min\{\tfrac{5}{2}s,s + 1\}$. \end{thm}

\begin{remark}\label{rem2} To explain the meaning of \eqref{form83} let us record some equivalent versions. Notice that if $p = 2n$ is an even integer, then 
\begin{displaymath} \|\hat{\sigma}\|_{L^{p}(B(R))}^{p} = \|\widehat{\sigma^{n}}\|_{L^{2}(B(R))}^{2}, \end{displaymath}
where $\sigma^{n} = \sigma \ast \cdots \ast \sigma$ is the $n$-fold convolution. If \eqref{form83} holds for all $R \geq 1$, then by Plancherel 
\begin{displaymath} \|\sigma^{n}_{\delta}\|_{L^{2}(\R^{2})}^{2} \lesssim_{s,t} \delta^{t - 2}, \qquad \delta > 0, \end{displaymath}
where $\sigma^{n}_{\delta} = (\sigma^{n}) \ast \varphi_{\delta}$, and $\{\varphi_{\delta}\}_{\delta > 0}$ is a standard approximate identity. This further implies $I_{\tau}(\sigma^{n}) \lesssim_{s,t} 1$ for all $\tau < t$, where $I_{\tau}(\nu)$ is the $\tau$-dimensional Riesz-energy (see \cite[Lemma 1.4]{MR4528124}). Conversely, if $I_{\tau}(\nu) \leq 1$, then 
\begin{displaymath} \|\hat{\nu}\|_{L^{2}(B(R))}^{2} \lesssim R^{2 - t}, \, R \geq 1 \quad \text{and} \quad \|\nu_{\delta}\|_{L^{2}(\R^{2})}^{2} \lesssim \delta^{t - 2}, \, \delta \in (0,1]. \end{displaymath}
Theorem \ref{main} can therefore be restated as follows: if $t \in [0,\min\{3s,s + 1\})$, then $I_{t}(\sigma^{n}) < \infty$ for $n \geq 1$ large enough. Moreover, if $s \geq 2/3$, then $I_{t}(\sigma \ast \sigma \ast \sigma) < \infty$ for all $t < s + 1$.
\end{remark}

\begin{remark} A weaker version of Theorem \ref{main2} follows directly from Theorem \ref{main}. In that version, $\Hd (nK)$ grows at an effective but non-explicit rate, determined by the dependence between the numbers $s,t$ and $p(s,t)$ in Theorem \ref{main}. For the details of this argument, see how \cite[Corollary 1.3]{MR4528124} is deduced from \cite[Theorem 1.1]{MR4528124} on \cite[p. 4]{MR4528124}. We will, however, give an independent argument for Theorem \ref{main2} in Section \ref{s2}, since this serves as a good warm-up for the proof of Theorem \ref{main}. \end{remark}

Theorem \ref{main} also yields upper bounds for the number of $\delta$-separated solutions to the quadratic Vinogradov system associated to "$s$-dimensional" subsets of the parabola:
\begin{cor}\label{cor2} For every $0 \leq s \leq 1$ and $t \in [0,\min\{3s,s + 1\})$, there exists $n(s,t) \in \N$ such that the following holds for all $n \geq n(s,t)$. Let $P \subset \mathbb{P}$ be a $\delta$-separated $(\delta,s)$-set. Then,
\begin{displaymath} |\{(p_{1},\ldots,p_{n},q_{1},\ldots,q_{n}) \in P^{2n} : |(p_{1} + \ldots + p_{n}) - (q_{1} + \ldots + q_{n})| \leq \delta\}| \leq \delta^{t}|P|^{2n}. \end{displaymath} \end{cor}

The notion of $(\delta,s)$-sets is introduced in Definition \ref{def:deltaSSet}, but this simply means that the normalised counting measure $\sigma = |P|^{-1}\mathcal{H}^{0}|_{P}$ satisfies the hypothesis of Theorem \ref{main} for scales $r \in [\delta,1]$. Consequently, Corollary \ref{cor2} is a straightforward application of Theorem \ref{main} to this measure $\sigma$. The details are recorded in Section \ref{s3}.
We also recall that, as in Theorem \ref{main}, the threshold $\min\{3s,s + 1\}$ for $t$ in the statement of Corollary \ref{cor2} is optimal, again by \cite[Example 1.8]{MR4528124}. As far as we know, it is possible that $n(s,t) = 3$ for all $0 \leq s \leq 1$ and $t \in [0,\min\{3s,s + 1\})$ (provided that $\delta > 0$ is small enough in terms of $s,t$).

Corollary \ref{cor2} should be compared with the result of Mudgal \cite[Corollary 1.2]{MR4535019}, which is a counterpart of Corollary \ref{cor2} in the discrete case. We also refer to \cite[Section 1.1]{MR4528124} for additional discussion. We thank Josh Zahl for bringing Mudgal's work into our attention.

\subsection{Proof ideas}\label{s:proofIdeas} In Theorem \ref{main2}, the first observation is that
\begin{displaymath} nK = (n - 1)K + K \end{displaymath}
contains many translates of $s$-dimensional subsets of the parabola $\mathbb{P}$. More precisely, writing $Z := (n - 1)K$, we have $nK \supset \bigcup\{z + K : z \in Z\}$. Note that 
\begin{displaymath} z + K \subset z + \mathbb{P} \quad \text{and} \quad \Hd (z + K) \geq s. \end{displaymath}
The second observation is that the diffeomorphism $\Psi(x,y) = (x,x^{2} - y)$ sends all these (subsets of) parabolas into (subsets of) lines, see Proposition \ref{prop3}. Since the map $z\mapsto \Psi(z+\mathbb{P})$ is locally bi-Lipschitz, it follows that $\Psi(nK) = \Psi(Z + K)$ is an \emph{$(s,t)$-Furstenberg set}, where $t := \Hd Z$ (see Remark \ref{rem1}). Finally, using the recent sharp result of Ren and Wang \cite{2023arXiv230808819R} on the dimension of Furstenberg sets, we may infer that 
\begin{displaymath} \Hd (nK) = \Hd \Psi(nK) \geq \min\left\{s + t,\tfrac{3s + t}{2},s + 1\right\}. \end{displaymath}
Keeping in mind that $t = \Hd Z = \Hd (n - 1)K$, the inequality above can be iterated to derive Theorem \ref{main2}, see Section \ref{s2} for the details.

The connection between Theorems \ref{main}-\ref{main2} and the Furstenberg set problem was already observed in \cite{MR4528124}. However, the reduction to the Furstenberg set problem in that paper was significantly more cumbersome than the one enabled by the map $\Psi$. The idea of transferring point-parabola incidences to point-line incidences using the map $\Psi$ was also employed by Pudl\'ak in \cite{MR2249270}.

The ideas outlined above are also present in the proof of Theorem \ref{main} (except the case $s \geq \tfrac{2}{3}$, discussed separately below). As mentioned in Remark \ref{rem2}, Theorem \ref{main} implies $I_{t}(\sigma^{n}) < \infty$ for all $t \in [0,\min\{3s,s + 1\})$, and $n \geq n(s,t)$. This yields $\Hd (nK) \geq t$, where $K = \spt \sigma$. Therefore, Theorem \ref{main} can be viewed as a quantitative version of Theorem \ref{main2}, with a poorer (and non-explicit) dependence of $n(s,t)$ on $s,t$.

Conversely, it turns out that Theorem \ref{main} can be deduced from the proof strategy of Theorem \ref{main2}, outlined above: the additional technical component is the \emph{$L^{2}$-flattening} method of Bourgain and Gamburd \cite{BourgainGamburd08}. This method is applied in a manner similar to the recent work of the first author with de Saxc\'e and Shmerkin \cite{2023arXiv230903068O}, see Corollary \ref{cor1}. This part of the proof is responsible for the (much) higher number of convolutions needed in Theorem \ref{main} than the number of sums needed in Theorem \ref{main2}.

Possibly a smaller number of convolutions and sums may suffice in Theorems \ref{main}-\ref{main2}:

\begin{question}\label{q1} Let $s \in [0,1]$, and let $\sigma$ be a Borel measure on $\mathbb{P}$ satisfying $\sigma(B(x,r)) \leq r^{s}$ for all $x \in \R^{2}$ and $r > 0$. Is it true that
\begin{displaymath} \|\hat{\sigma}\|_{L^{6}(B(R))} \lesssim_{\epsilon,s} R^{(2 - \min\{3s,s + 1\})/6 + \epsilon}, \qquad R \geq 1, \, \epsilon > 0? \end{displaymath}
Equivalently, is $I_{t}(\sigma \ast \sigma \ast \sigma) < \infty$ for $t \in [0,\min\{3s,s + 1\})$? \end{question} 
The fractal geometric counterpart of this question proposes that if $K \subset \mathbb{P}$ is a Borel set with $\Hd K = s$, then $\Hd (K + K + K) \geq \min\{3s,s + 1\}$. This problem can be viewed as the continuum analogue of a question Bourgain and Demeter \cite[Question 2.13]{MR3374964}.

\subsection{The range $s \geq \tfrac{2}{3}$} Theorem \ref{main} implies a positive answer to Question \ref{q1} when $s \geq \tfrac{2}{3}$. In this range, the proof of Theorem \ref{main} differs significantly from the proof outline described in the previous section. In fact, the main technical tool is the following $L^{2}$-Sobolev smoothing estimate:

\begin{proposition}\label{mainProp} Let $s \in (\tfrac{1}{2},1]$ and $t \in (0,2)$. Let $\mu$ be a Borel measure on $B(1)$ with $I_{t}(\mu) \leq 1$, and let $\sigma$ be a Borel measure on $\mathbb{P}$ satisfying $\sigma(B(x,r)) \leq r^{s}$ for all $x \in \R^{2}$ and $r > 0$. Then, $I_{\zeta}(\mu \ast \sigma) \lesssim_{\zeta} 1$ for all $0 \leq \zeta < \zeta(s,t)$, where $\zeta(s,t) = \min\{t + (2s - 1),s + 1\}$.
\end{proposition}
The case $s \geq \tfrac{2}{3}$ of Theorem \ref{main} follows so easily from this proposition that we record the details immediately. 
 
\begin{proof}[Proof of Theorem \ref{main} in the case $s \geq \tfrac{2}{3}$] Fix $s \in [\tfrac{2}{3},1]$, and let $\sigma$ be a Borel measure on $\mathbb{P}$ satisfying $\sigma(B(x,r)) \leq r^{s}$. The first part of \cite[Theorem 1.1]{MR4528124} implies $I_{2s - \epsilon}(\sigma \ast \sigma) \lesssim_{\epsilon} 1$ for all $\epsilon > 0$. Next, Proposition \ref{mainProp} implies $I_{\zeta(s,2s - \epsilon) - \epsilon}(\sigma \ast \sigma \ast \sigma) \lesssim_{\epsilon} 1$. Finally, note that since $s \geq \tfrac{2}{3}$, one has $\zeta(s,2s) = s + 1$, and the number $\zeta(s,2s - \epsilon) - \epsilon$ can be made arbitrarily close to $s + 1$ by letting $\epsilon \to 0$. This yields \eqref{form83} with $p = 6$, recalling the equivalences discussed in Remark \ref{rem2}. \end{proof}

Finally, a natural open question concerns the generalisation of Theorems \ref{main}, \ref{main2}, and Proposition \ref{mainProp} to subsets of curves (or even graphs) more general than the parabola. In \cite{MR4661079}, Demeter and Dasu prove a variant of Theorem \ref{main}, where $p = 6$ and $t \in [0,2s + \beta)$ for some $\beta > 0$, but the parabola $\mathbb{P}$ can be replaced by any graph of a $C^{3}$-function $\gamma \colon [-1,1] \to \R$ satisfying 
\begin{equation}\label{form22} \inf \{|\ddot{\gamma}(x)| : x \in [-1,1]\} > 0. \end{equation}
For the proof of Theorem \ref{main} in the present paper to work for more general graphs than $\mathbb{P}$, we would need the following variant of the Furstenberg set theorem:
\begin{question}\label{q2} Let $s \in (0,1]$ and $t \in [0,2]$. Assume that $\Gamma$ is the graph of a $C^{\infty}$-function, or possibly a $C^{2}$-function, $\gamma \colon [-1,1] \to \R$ satisfying \eqref{form22}. Let $F \subset \R^{2}$ be a set with the following property: there exists another set $K \subset \R^{2}$ with $\Hd K \geq t$ such that
\begin{displaymath} \Hd (F \cap (z + \Gamma)) \geq s, \qquad z \in K. \end{displaymath}
Is it true that $\Hd F \geq \min\{s + t,(3s + t)/2,s + 1\}$?
\end{question}

This is true if $\Gamma = \mathbb{P}$ as follows from the proof of Theorem \ref{main2} in Section \ref{s2}. In fact, the proof goes through for any $\Gamma$ which admits a diffeomorphism $\Psi$ such that $\Psi(z+\Gamma)$ is a line, for every $z\in\R^2$. A characterization of such pairs $(\Gamma, \Psi)$ was performed in \cite{MR4568803}: For example, also the graphs of $x\mapsto \log x$ and $x\mapsto e^{-x}$ admit such a function $\Psi$. We thank the referee for bringing the work \cite{MR4568803} into our attention. 

\subsection{Outline of the paper} Section \ref{s2} contains the proof of Theorem \ref{main2}. This is the technically easiest result, but the argument already contains many of the ingredients needed for the proof of Theorem \ref{main}. The proof of Theorem \ref{main} is conducted in Sections \ref{s:prelim}-\ref{s3}. Section \ref{s3} also contains the proof of Corollary \ref{cor2}. Finally, the $L^{2}$-Sobolev smoothing estimate in Proposition \ref{mainProp} is established in Section \ref{s5}.

\begin{notation}
	We write $f\lesssim g$ if there exists an absolute constant $C>0$ such that $f\le C g$. If $C$ depends on some parameter $\epsilon$, we will write $f\lesssim_{\epsilon} g$. In case $f\lesssim g\lesssim f$ we write $f\sim g$, while $f\sim_{\epsilon} g$ denotes $f\lesssim_{\epsilon} g\lesssim_{\epsilon} f$.
	If $\mathcal F$ is a collection of subsets of $\mathbb R^2$, we denote $\cup \mathcal F:= \bigcup_{Q\in\mathcal F}Q\subseteq \mathbb R^2.$
\end{notation}

\subsection*{Acknowledgements} We thank the reviewers for a careful reading of the manuscript, and for many useful comments.

\section{Proof of Theorem \ref{main2}}\label{s2}

We need two tools for the proof of Theorem \ref{main2}. The first one is the resolution by Ren and Wang \cite{2023arXiv230808819R} of the \emph{Furstenberg set conjecture}, proposed by Wolff \cite{MR1692851,Wolff99} in the late 90s.
\begin{thm}\label{t:renWangQualitative} Let $s \in (0,1]$ and $t \in [0,2]$. Let $\mathcal{L}$ be a set of lines in $\R^{2}$ with $\Hd \mathcal{L} \geq t$, and assume that $K \subset \R^{2}$ is a set satisfying $\Hd (K \cap \ell) \geq s$ for all $\ell \in \mathcal{L}$. Then,
\begin{equation}\label{form10} \Hd K \geq \min\left\{s + t,\tfrac{3s + t}{2},s + 1\right\}. \end{equation} \end{thm} 
\begin{remark}\label{rem1} Sets $K \subset \R^{2}$ with the properties stated in Theorem \ref{t:renWangQualitative} are called \emph{$(s,t)$-Furstenberg sets}. \end{remark}

There are several equivalent ways to define $\Hd \mathcal{L}$. Perhaps the most elegant one is to place some natural metric on the affine Grassmannian $\mathcal{A}(2,1)$, consisting of all lines in $\R^{2}$, and define the Hausdorff dimension of subsets of $\mathcal{A}(2,1)$ using this metric. A concrete choice for the metric is
\begin{equation}\label{form14} d_{\mathcal{A}(2,1)}(\ell_{1},\ell_{2}) := \|\pi_{L_{1}} - \pi_{L_{2}}\| + |a_{1} - a_{2}|, \qquad \ell_{j} = a_{j} + L_{j}. \end{equation}
Here $L_{j} \in \mathcal{G}(2,1)$ is a $1$-dimensional subspace of $\R^{2}$, and $a_{j} \in L_{j}^{\perp}$.

A second common way to define $\Hd \mathcal{L}$ is to parametrise non-vertical lines of $\R^{2}$ by 
\begin{equation}\label{form20} \ell(a,b) := \{(x,y) : y = ax + b\}, \qquad (a,b) \in \R^{2}, \end{equation}
and set $\Hd \mathcal{L} := \Hd \{(a,b) \in \R^{2} : \ell(a,b) \in \mathcal{L}\}$. This definition has the advantage of being very explicit, and the caveat of only making sense if $\mathcal{L}$ consists of non-vertical lines (the range of $(a,b) \mapsto \ell(a,b)$). For families of non-vertical lines, the two notions of dimension coincide: it is easy to check that the map $(a,b) \mapsto \ell(a,b)$ is locally bi-Lipschitz $\R^{2} \to (\mathcal{A}(2,1),d_{\mathcal{A}(2,1)})$.

For the proof of Theorem \ref{main}, we will eventually need a more quantitative version of Theorem \ref{t:renWangQualitative} (stated as Theorem \ref{thm:renWang}), but the qualitative version above is sufficient for Theorem \ref{main2}. For the proof of Theorem \ref{main2}, the second main tool is the map $\Psi \colon \R^{2} \to \R^{2}$,
\begin{displaymath} \Psi(x,y) := (x,x^{2} - y), \qquad (x,y) \in \R^{2}. \end{displaymath}
\begin{proposition}\label{prop3} The map $\Psi$ has the following properties:
\begin{itemize}
\item[(a)] $\Psi \circ \Psi = \mathrm{Id}$. In particular, $\Psi$ is invertible and locally bi-Lipschitz.
\item[(b)] For $(x_{0},y_{0}) \in \R^{2}$, define the line $\ell_{(x_{0},y_{0})} = \{(x,y) : y = y_{0} + 2x_{0}(x - x_{0})\}$. Then,
\begin{displaymath} \Psi(\ell_{(x_{0},y_{0})}) = \Psi(x_{0},y_{0}) + \bar{\mathbb{P}}, \end{displaymath}
where $\bar{\mathbb{P}} = \{(x,x^{2}) : x \in \R\}$ is the full upward pointing parabola. Conversely, by \textup{(a)}, $\Psi$ maps the translated parabola $\Psi(x_{0},y_{0}) + \bar{\mathbb{P}}$ onto $\ell_{(x_{0},y_{0})}$.
\item[(c)] The map $z \mapsto \Psi(z + \bar{\mathbb{P}})$ is locally bi-Lipschitz $\R^{2} \to (\mathcal{A}(2,1),d_{\mathcal{A}(2,1)})$, recall \eqref{form14}. In particular: if $\mathcal{P} := \{z + \bar{\mathbb{P}} : z \in Z\}$ is a family of translated parabolas indexed by a set $Z \subset \R^{2}$, then the line family $\Psi(\mathcal{P})$ satisfies $\Hd \Psi(\mathcal{P}) = \Hd Z$.
\end{itemize} \end{proposition}

\begin{proof} To prove (a), note that $\Psi(x,x^{2} - y) = (x,x^{2} - (x^{2} - y)) = (x,y)$. To prove (b):
\begin{displaymath} \Psi(x,y_{0} + 2x_{0}(x - x_{0})) = (x,x^{2} - y_{0} - 2xx_{0} + 2x_{0}^{2}) = \Psi(x_{0},y_{0}) + (x - x_{0},(x - x_{0})^{2}) \end{displaymath} 
for $(x_{0},y_{0}) \in \R^{2}$ and $x \in \R$.

To prove (c), note that $\Psi(z + \bar{\mathbb{P}}) = \Psi^{-1}(z + \bar{\mathbb{P}}) = \ell_{\Psi(z)}$ by (a)-(b). On the other hand, writing $z = (x_{0},y_{0})$, we have 
\begin{displaymath} \ell_{\Psi(z)} = \ell_{(x_{0},x_{0}^{2} - y_{0})} = \{(x,y) : y = (x_{0}^{2} - y_{0}) + 2xx_{0} - 2x_{0}^{2}\} = \ell(2x_{0},-x_{0}^{2} - y_{0})\end{displaymath}
in the notation \eqref{form20}. Therefore, $z \mapsto \Psi(z + \bar{\mathbb{P}})$ is the composition of the locally bi-Lipschitz maps $(x_{0},y_{0}) \mapsto (2x_{0},-x_{0}^{2} - y_{0})$ and $(a,b) \mapsto \ell(a,b)$. \end{proof}

We are then equipped to prove Theorem \ref{main2}.

\begin{proof}[Proof of Theorem \ref{main2}] Recall that $K \subset \mathbb{P}$, and $\Hd K \geq s$. We claim the following:
\begin{equation}\label{form11} \Hd nK \geq \min\left\{\tfrac{3s + \Hd [(n - 1)K]}{2},s + 1\right\}, \qquad n \geq 2, \end{equation} 
where $mK = K + \ldots + K$ refers to the $m$-fold sum-set. To prove \eqref{form11}, note that $nK = Z + K$, where $Z := (n - 1)K$. The set $nK$, therefore, has the property 
\begin{displaymath} \Hd [(nK) \cap (z + \mathbb{P})] \geq \Hd (z + K) \geq s, \qquad z \in Z. \end{displaymath}
Since $\Psi$ is a dimension-preserving bijection, it follows that
\begin{displaymath} \Hd [\Psi(nK) \cap \Psi(z + \mathbb{P})] \geq s, \qquad z \in Z. \end{displaymath}
By Proposition \ref{prop3}(b)-(c), the inequality above shows that $F := \Psi(nK)$ is an $(s,\Hd Z)$-Furstenberg set: there exists a $(\Hd Z)$-dimensional family of lines, namely 
\begin{displaymath} \mathcal{L} = \{\Psi(z + \mathbb{P}) : z \in Z\} \subset \mathcal{A}(2,1), \end{displaymath}
with the property that $\Hd (F \cap \ell) \geq s$ for all $\ell \in \mathcal{L}$. Writing $t := \Hd Z \geq s$, we may therefore deduce from \eqref{form10} that
\begin{displaymath} \Hd (nK) = \Hd \Psi(nK) \geq \min\left\{s + t,\tfrac{3s + t}{2},s + 1\right\}. \end{displaymath}
This implies \eqref{form11}, because $s + t \leq (3s + t)/2$ only when $s \geq t$.

Finally, \eqref{form11} implies Theorem \ref{main2}, because if $\{t_{n}\}_{n \in \N}$ is a sequence of real numbers satisfying $t_{n} \geq \min\{(3s + t_{n - 1})/2,s + 1\}$ for $n \geq 2$, with initial condition $t_{1} \geq s$, then a straightforward induction shows that
\begin{equation}\label{form12} t_{n} \geq \min\{3s - s \cdot 2^{-(n - 2)},s + 1\}, \qquad n \geq 1. \end{equation}
This completes the proof of Theorem \ref{main2}. \end{proof}


\section{Preliminaries for the proof of Theorem \ref{main}}\label{s:prelim}

We start by setting some notation and terminology.

\begin{definition}[$(\delta,s,C)$-set]\label{def:deltaSSet} Let $s \in [0,d]$, $C > 0$. A set $P \subset \R^{d}$ is called a \emph{$(\delta,s,C)$-set} if
\begin{equation}\label{form15} |P \cap B(x,r)|_{\delta} \leq Cr^{s}|P|_{\delta}, \qquad x \in \R^{d}, \, r \geq \delta. \end{equation} 
Here $|\cdot|_{\delta}$ refers to the $\delta$-covering number. A line set $\mathcal{L} \subset \mathcal{A}(2,1)$ is called a $(\delta,s,C)$-set if it satisfies the counterpart of \eqref{form15} relative to the metric $d_{\mathcal{A}(2,1)}$ defined in \eqref{form14}.
 \end{definition} 
 
From now on,  $\mathcal{D}_{\delta}(\R^{d})$ refers to all dyadic cubes in $\R^{d}$ of side-length $\delta \in 2^{-\N}$, and 
\begin{displaymath} \mathcal{D}_{\delta} := \{p \in \mathcal{D}_{\delta}(\R^{d}) : p \subset [0,1)^{d}\}. \end{displaymath}
If $K \subset \R^{d}$ is a set, we also write $\mathcal{D}_{\delta}(K) := \{p \in \mathcal{D}_{\delta}(\R^{d}) : K \cap p \neq \emptyset\}$. A family $\mathcal{P} \subset \mathcal{D}_{\delta}(\R^{d})$ is called a $(\delta,s,C)$-set if its union $\cup \mathcal{P}$ is a $(\delta,s,C)$-set in the sense of Definition \ref{def:deltaSSet}. In the proof of Theorem \ref{main}, the Furstenberg set theorem of Ren-Wang enters through the next lemma, which is a $\delta$-discretised counterpart of \eqref{form11}:

\begin{lemma}\label{lemma1} For all $s \in (0,1]$, $t \in [0,2]$, and $\kappa > 0$, there exist $\epsilon = \epsilon(s,t,\kappa) > 0$ and $\delta_{0} = \delta_{0}(s,t,\kappa,\epsilon) > 0$ such that the following holds for all $\delta \in (0,\delta_{0}]$. Let $\mathcal{P} \subset \mathcal{D}_{\delta}$ be a non-empty $(\delta,t,\delta^{-\epsilon})$-set. Assume that to every $p \in \mathcal{P}$ there is an associated non-empty $(\delta,s,\delta^{-\epsilon})$-set $\mathcal{F}(p) \subset \mathcal{D}_{\delta}$ with the property 
\begin{equation}\label{form13} q \cap (p + \mathbb{P}) \neq \emptyset, \qquad q \in \mathcal{F}(p), \, p \in \mathcal{P}. \end{equation}
Here $p + \mathbb{P}$ is the translate of $\mathbb{P}$ by the centre of $p$. Then, the union $\mathcal{F} := \bigcup \{\mathcal{F}(p) : p \in \mathcal{P}\}$ satisfies
\begin{displaymath} |\mathcal{F}| \geq \delta^{-\gamma(s,t) + \kappa}, \qquad \gamma(s,t) := \min\left\{s + t,\tfrac{3s + t}{2},s + 1\right\}. \end{displaymath}
\end{lemma}

Using the map $\Psi$ from Proposition \ref{prop3}, Lemma \ref{lemma1} turns out to be a reformulation of a $\delta$-discretised version of Ren and Wang's result (\cite[Theorem 4.1]{2023arXiv230808819R}) stated below as Theorem \ref{thm:renWang}. First, let us explain what is meant by a $(\delta,t,C)$-set of tubes:
\begin{definition}[$(\delta,s,C)$-set of tubes] Let $\mathcal{T}$ be a family of $\delta$-tubes. Thus, 
\begin{displaymath} \mathcal{T} = \{[\ell]_{\delta/2} : \ell \in \mathcal{L}\}, \end{displaymath}
where $[A]_{r}$ refers to the $r$-neighbourhood of a set $A \subset \R^{d}$, and $\mathcal{L} \subset \mathcal{A}(2,1)$. We say that $\mathcal{T}$ is a $(\delta,s,C)$-set if the line family $\mathcal{L}$ is a $(\delta,s,C)$-set in the sense of Definition \ref{def:deltaSSet}. \end{definition}

\begin{thm}[Ren-Wang]\label{thm:renWang}  For all $s \in (0,1]$, $t \in [0,2]$, and $\kappa > 0$, there exist $\epsilon = \epsilon(s,t,\kappa) > 0$ and $\delta_{0} = \delta_{0}(s,t,\kappa,\epsilon) > 0$ such that the following holds for all $\delta \in (0,\delta_{0}]$. Let $\mathcal{T} \subset \mathcal{T}^{\delta}$ be a non-empty $(\delta,t,\delta^{-\epsilon})$-set. Assume that to every $T \in \mathcal{T}$ there is an associated non-empty $(\delta,s,\delta^{-\epsilon})$-set $\mathcal{P}(T) \subset \mathcal{D}_{\delta}$ with the property 
\begin{displaymath} q \cap T \neq \emptyset, \qquad q \in \mathcal{P}(T), \, T \in \mathcal{T}. \end{displaymath}
Then, the union $\mathcal{P} := \bigcup \{\mathcal{P}(T) : T \in \mathcal{T}\}$ satisfies
\begin{displaymath} |\mathcal{P}| \geq \delta^{-\gamma(s,t) + \kappa}, \qquad \gamma(s,t) := \min\left\{s + t,\tfrac{3s + t}{2},s + 1\right\}. \end{displaymath}
\end{thm}

To be accurate, Theorem \ref{thm:renWang} is a "dual version" of \cite[Theorem 4.1]{2023arXiv230808819R}. For an explanation about how to pass between the two versions of the statement, see the proof of \cite[Theorem 3.2]{2023arXiv230808819R}.  We will now deduce Lemma \ref{lemma1} from Theorem \ref{thm:renWang}:
\begin{proof}[Proof of Lemma \ref{lemma1}] 

The map $\Psi$ from Proposition \ref{prop3} allows us to pass easily between Theorem \ref{thm:renWang} and Lemma \ref{lemma1}. Indeed, let $\mathcal{P}$ and $\{\mathcal{F}(p)\}_{p \in \mathcal{P}}$ be the objects from Lemma \ref{lemma1}. Fix $p \in \mathcal{P}$, and write (the centre) $p = \Psi(x_{p},y_{p})$ for some $(x_{p},y_{p}) \in \R^{2}$. Recall from \eqref{form13} that $\mathcal{F}(p)$ is a non-empty $(\delta,s,\delta^{-\epsilon})$-set of dyadic $\delta$-squares, all of which intersect $p + \mathbb{P}$. In particular $\cup \mathcal{F}(p) \subset [\Psi(x_{p},y_{p}) + \mathbb{P}]_{2\delta}$, where $[A]_{r}$ is the $r$-neighbourhood of $A \subset \R^{2}$. Therefore, by Proposition \ref{prop3}(b), and the local bi-Lipschitz property of $\Psi^{-1} = \Psi$,
\begin{equation}\label{form3} \Psi^{-1}(\cup \mathcal{F}(p)) \subset \Psi^{-1}([\Psi(x_{p},y_{p}) + \mathbb{P}]_{2\delta}) \subset [\ell_{(x_{p},y_{p})}]_{C\delta}, \end{equation}
for some absolute $C > 0$. Here $\ell_{(x_{p},y_{p})} \in \mathcal{A}(2,1)$ is the line introduced in Proposition \ref{prop3}(b). Next, recall that $\mathcal{P}$ is a $(\delta,t,\delta^{-\epsilon})$-set. This implies, using Proposition \ref{prop3}(c), that the line family
\begin{displaymath} \mathcal{L} := \{\Psi((x_{p},y_{p}) + \bar{\mathbb{P}}) : p \in \mathcal{P}\} = \{\ell_{(x_{p},y_{p})} : p \in \mathcal{P}\} \end{displaymath}
is a $(\delta,t,O(\delta^{-\epsilon}))$-set. Therefore,
\begin{displaymath} \mathcal{T} := \{T_{p} : p \in \mathcal{P}\} := \{[\ell_{(x_{p},y_{p})}]_{C\delta} : p \in \mathcal{P}\} \end{displaymath}
is a $(\delta,t,O(\delta^{-\epsilon}))$-set of $(C\delta)$-tubes. According to \eqref{form3}, the tube $T_{p}$ contains the $(\delta,s,O(\delta^{-\epsilon}))$-set $\Psi^{-1}(\cup \mathcal{F}(p))$. After small technical adjustments we rather omit (covering $(C\delta)$-tubes with $\delta$-tubes and replacing $\Psi^{-1}(\cup \mathcal{F}(p))$ with unions of dyadic $\delta$-squares), Theorem \ref{thm:renWang} is now applicable to the objects $\mathcal{T}$ and $\{\Psi^{-1}(\cup \mathcal{F}(p))\}_{p \in \mathcal{P}}$. The conclusion is that
\begin{displaymath} \left| \bigcup_{p \in \mathcal{P}} \mathcal{F}(p) \right| \sim \left| \bigcup_{p \in \mathcal{P}} \Psi^{-1}(\cup \mathcal{F}(p)) \right|_{\delta} \geq \delta^{-\gamma(s,t) + \kappa},\end{displaymath}
using the (local) bi-Lipschitz property of $\Psi^{-1} = \Psi$ in the first equation. This completes the proof of Lemma \ref{lemma1}. \end{proof}

\section{$L^{2}$-flattening}\label{s4}

The main results in this section are Proposition \ref{prop2} and Corollary \ref{cor1}. We call these "$L^{2}$-flattening" results: they roughly state that the convolution of fractal measures $\mu$ and $\sigma$ (with $\sigma$ supported on the parabola $\mathbb{P}$) has a smoothing effect on Riesz-energies, which is an $L^{2}$-based quantity. The terminology "$L^{2}$-flattening" may be slightly misleading, since the result we aim for is not of the form $I_{t}(\mu \ast \sigma) < I_{t}(\mu)$, that is, a reduction of the Riesz energy with a fixed exponent. Rather, Proposition \ref{prop2} goes in the direction of showing that the Riesz energy of $\mu \ast \sigma$ is finite with some exponent $t' > t$, provided that $I_{t}(\mu) < \infty$ and $I_{s}(\sigma) < \infty$ for (sufficiently large) $s > 0$. More accurately, Proposition \ref{prop2} falls short of proving what is stated above, but rather accomplishes this for a sufficiently high self-convolution power $(\mu \ast \sigma)^{k}$. 

The "$L^{2}$-flattening" terminology is historically motivated. The proof of Proposition \ref{prop2} is modelled after the proof of \cite[Lemma 3.1]{2023arXiv230903068O}, which further borrowed many ideas from Bourgain's finite field paper \cite{MR2481734}. All these arguments apply a strategy introduced by Bourgain and Gamburd, see \cite[Proposition 2.2]{BourgainGamburd08} and \cite[Proposition 2]{MR2415383}, which is nicknamed the "$\ell^{2}$-flattening lemma" by the authors in \cite{BourgainGamburd08}.

From now on, if $\mu$ is a finite measure on $\R^{2}$, the notation $\mu^{k}$ refers to the $k$-fold self-convolution $\mu \ast \cdots \ast \mu$. If moreover $\delta \in (0,1]$, and $s \in (0,d)$, the notation $I_{s}^{\delta}(\mu)$ refers to the $s$-dimensional Riesz-energy of the mollified measure $\mu_{\delta} = \mu \ast \psi_{\delta}$, thus
\begin{displaymath} I_{s}^{\delta}(\mu) = I_{s}(\mu_{\delta}) = \iint \frac{\mu_{\delta}(x) \mu_{\delta}(y)}{|x - y|^{s}} \, dx \, dy. \end{displaymath}
Here $\{\psi_{\delta}\}_{\delta > 0}$ is any standard approximate identity of the form $\psi_{\delta}(x) = \delta^{-2}\psi(x/\delta)$, where $\int \psi = 1$, and $\psi$ is radially decreasing, and $\mathbf{1}_{B(1/2)} \leq \psi \leq \mathbf{1}_{B(1)}$.

We start by recording the following elementary lemma.

\begin{lemma}\label{lemma2} Let $\epsilon \in (0,1)$, $C,d \geq 1$, and $s \in (0,d)$. Then the following holds for all $\delta \in 2^{-\N}$ sufficiently small.

Let $\mu$ be a Borel probability measure on $\R^{d}$ such that $I_{s}^{\delta}(\mu) \leq C$. Let $K \subset B(1)$ be a Borel set such that $\mu(K) \geq \delta^{\epsilon}$. Then, there exists a non-empty $(\delta,s,O_{d}(C\delta^{-3\epsilon}))$-set $\mathcal{P} \subset \mathcal{D}_{\delta}(K)$. \end{lemma}

\begin{proof} In this proof, the $\lesssim$ notation may hide constants depending on the dimension $d$. First note that $\mu_{\delta}(K_{\delta}) \gtrsim \delta^{\epsilon}$, where $K_{\delta} := \cup \mathcal{D}_{\delta}(K)$ is the dyadic $\delta$-neighbourhood of $K$. This is because
\begin{displaymath} \mu_{\delta}(K_{\delta}) = \int \mathbf{1}_{K_{\delta}}(z) (\mu \ast \psi_{\delta})(z) \, dz = \int \mathbf{1}_{K_{\delta}} \ast \psi_{\delta} \, d\mu, \end{displaymath} 
and $\mathbf{1}_{K_{\delta}} \ast \psi_{\delta} \gtrsim \mathbf{1}_{K}$. Next, Chebychev's inequality implies that
\begin{displaymath} B_{A} := \{z \in K_{\delta} : \int |z - w|^{-s} \, d\mu_{\delta}(w) \geq AC\delta^{-\epsilon}\}, \qquad A \geq 1, \end{displaymath} 
satisfies $\mu_{\delta}(B_{A}) \leq \delta^{\epsilon}/(AC) \cdot I_{s}^{\delta}(\mu) \leq \delta^{\epsilon}/A$. Therefore, provided that $A \geq 1$ is a sufficiently large absolute constant, $\mu_{\delta}(K_{\delta} \, \setminus \, B_{A}) \geq \tfrac{1}{2}\mu_{\delta}(K_{\delta}) \gtrsim \delta^{\epsilon}$. Now the restriction 
\begin{displaymath} \mu_{\delta}' := \mu_{\delta}|_{(K_{\delta} \, \setminus \, B_{A})} \end{displaymath}
satisfies $\mu_{\delta}'(K_{\delta}) \gtrsim \delta^{\epsilon}$ and 
\begin{equation}\label{form24} \mu_{\delta}'(B(z,r)) \lesssim C\delta^{-\epsilon}r^{s}, \qquad z \in \R^{2}, \, r > 0. \end{equation}
To proceed finding a $(\delta,s,O_{d}(C\delta^{-3\epsilon}))$-set $\mathcal{P} \subset \mathcal{D}_{\delta}(K)$, we partition $\mathcal{D}_{\delta}(K)$ according to the value of $\mu_{\delta}'(p) \in [0,1]$:
\begin{displaymath} \mathcal{D}_{\delta}(K) = \bigcup_{j \in \N} \mathcal{D}_{\delta}^{j}(K), \qquad \mathcal{D}_{\delta}^{j}(K) := \{p \in \mathcal{D}_{\delta}(K) : 2^{-j - 1} \leq \mu_{\delta}'(p) \leq 2^{-j}\}. \end{displaymath}
We may disregard values of $j \in \N$ with $2^{j} \geq \delta^{-d - 1}$, because the total $\mu_{\delta}'$ measure covered by cubes with these indices is so small (recall here that $K \subset B(1)$):
\begin{displaymath} \sum_{2^{j} \geq \delta^{-d - 1}} \mu_{\delta}'(\cup \mathcal{D}^{j}_{\delta}(K)) \lesssim \sum_{2^{j} \geq \delta^{-d - 1}} 2^{-j}\delta^{-d} \lesssim \delta. \end{displaymath}
Since $\mu_{\delta}'(K_{\delta}) \gtrsim \delta^{\epsilon}$, and $\epsilon \in (0,1)$, the total measure on the left hand side above is at most $\tfrac{1}{2}\mu_{\delta}'(K_{\delta})$ for $\delta > 0$ sufficiently small.

Finally, note that there are $\leq (d + 1)\log(1/\delta)$ indices $j \in \N$ with $1 \leq 2^{j} \leq \delta^{-d - 1}$. So, there exists an index $j \in \N$ with $1 \leq 2^{j} \leq \delta^{-d - 1}$ such that, writing $\mathcal{P} := \mathcal{D}_{\delta}^{j}(K)$, we have
\begin{displaymath} \mu_{\delta}'(\cup \mathcal{P}) \gtrsim (\log(1/\delta))^{-1}\mu_{\delta}'(K_{\delta}) \gtrsim \delta^{2\epsilon}. \end{displaymath}
We now claim that $\mathcal{P}$ is a $(\delta,s,O_{d}(C\delta^{-3\epsilon}))$-set. Let $\delta \leq \Delta \leq 1$, and $Q \in \mathcal{D}_{\Delta}(\R^{d})$. Then, using the definition of $\mathcal{D}_{\delta}^{j}(K)$,
\begin{displaymath} 2^{-j} \cdot |\mathcal{P} \cap Q| \lesssim \mu_{\delta}'(Q) \stackrel{\eqref{form24}}{\lesssim} C\delta^{-\epsilon}\Delta^{s}, \end{displaymath}
therefore $|\mathcal{P} \cap Q| \lesssim C\delta^{-\epsilon}\Delta^{s} \cdot 2^{j}$. On the other hand,
\begin{displaymath} \delta^{2\epsilon} \lesssim \mu_{\delta}'(\cup \mathcal{P}) \lesssim 2^{-j}|\mathcal{P}|, \end{displaymath}
so $2^{j} \lesssim \delta^{-2\epsilon}|\mathcal{P}|$. We have shown that $|\mathcal{P} \cap Q| \lesssim C\delta^{-3\epsilon}\Delta^{s}|\mathcal{P}|$ for all $Q \in \mathcal{D}_{\Delta}(\R^{d})$. In other words $\mathcal{P}$ is a $(\delta,s,O(C\delta^{-3\epsilon}))$-set, as claimed.\end{proof}

\begin{proposition}\label{prop1} For all $s \in (0,1]$, $t \in [0,2]$, and $\kappa > 0$, there exist $\epsilon = \epsilon(\kappa,s,t) > 0$ and $\delta_{0} = \delta_{0}(\kappa,s,t,\epsilon) > 0$ such that the following holds for all $\delta \in (0,\delta_{0}]$ and $k \geq 1$. Assume that $\mu,\sigma$ are Borel probability measures with 
\begin{displaymath} \spt \mu \subset B(1) \subset \R^{2} \quad \text{and} \quad \spt \sigma \subset \mathbb{P}. \end{displaymath}
Assume additionally that $I^{\delta}_{t}(\mu) \leq \delta^{-\epsilon}$ and $I^{\delta}_{s}(\sigma) \leq \delta^{-\epsilon}$. Write $\Pi := \mu \ast \sigma$, and assume that $E \subset \R^{2}$ is a Borel set with $\Pi^{k}(E) \geq \delta^{\epsilon}$. Then,
\begin{equation}\label{form1} |E|_{\delta} \geq \delta^{-\gamma(s,t) + \kappa}, \qquad \gamma(s,t) := \min\left\{s + t,\tfrac{3s + t}{2},s + 1\right\}. \end{equation} \end{proposition}

\begin{proof} First of all, we may assume $k = 1$, because for $k \geq 2$, we may write
\begin{displaymath} \Pi^{k}(E) = \int \Pi(E - z_{2} - \ldots - z_{k}) \, d\Pi(x_{1})\cdots d\Pi(x_{k}). \end{displaymath}
Thus, if $\Pi^{k}(E) \geq \delta^{\epsilon}$, then also $\Pi(E - z_{2} - \ldots - z_{k}) \geq \delta^{\epsilon}$ for some $z_{2},\ldots,z_{k} \in \R^{2}$, and $|E|_{\delta} = |E - z_{2} - \ldots - z_{k}|_{\delta}$.

Let us then consider the case $k = 1$. By hypothesis,
\begin{displaymath} \delta^{\epsilon} \leq (\mu \ast \sigma)(E) = \int \sigma(E - z) \, d\mu (z). \end{displaymath}
Taking into account that both $\mu,\sigma$ are probability measures, there exists a set $K \subset \spt \mu$ with $\mu(K) \geq \delta^{\epsilon}$ such that $\sigma(E - z) \geq \delta^{\epsilon}$ for all $z \in K$. Now, recalling that $I^{\delta}_{t}(\mu) \leq \delta^{-\epsilon}$, Lemma \ref{lemma2} implies the the existence of a non-empty $(\delta,t,\delta^{-O(\epsilon)})$-set $\mathcal{P} \subset \mathcal{D}_{\delta}(K)$.

For every $p \in \mathcal{P}$, fix a distinguished point $z_{p} \in K \cap p$. Then, using $\sigma(E - z_{p}) \geq \delta^{\epsilon}$ and $I^{\delta}_{s}(\sigma) \leq \delta^{-\epsilon}$, one may use Lemma \ref{lemma2} a second time to find a non-empty $(\delta,s,\delta^{-O(\epsilon)})$-set 
\begin{displaymath} \mathcal{F}' \subset \mathcal{D}_{\delta}(E - z_{p}) \cap \mathcal{D}_{\delta}(\spt \sigma). \end{displaymath}
Since $\spt \sigma \subset \mathbb{P}$, this implies that every square in $z_{p} + \mathcal{F}'$ intersects both $E$ and $z_{p} + \mathbb{P}$. The family $z_{p} + \mathcal{F}'$ does not exactly consist of dyadic squares: to make Lemma \ref{lemma1} literally applicable, we replace $z_{p} + \mathcal{F}'$ by a family $\mathcal{F}(p) \subset \mathcal{D}_{\delta}$ with the properties
\begin{itemize}
\item $\cup \mathcal{F}(p) \subset \cup \mathcal{D}_{8\delta}(z_{p} + \mathcal{F}') \subset \cup \mathcal{D}_{8\delta}(E)$ and $|\mathcal{F}(p)| \sim |z_{p} + \mathcal{F}'|_{\delta}$,
\item $q \cap (p + \mathbb{P}) \neq \emptyset$ for all $q \in \mathcal{F}(p)$, $p \in \mathcal{P}$.
\end{itemize}
Then each $\mathcal{F}(p)$ is a non-empty $(\delta,s,\delta^{-O(\epsilon)})$-set, and satisfies the requirements of Lemma \ref{lemma1}. If $\epsilon > 0$ is small enough (depending on $\kappa,s,t$ only), Lemma \ref{lemma1} now implies that
\begin{displaymath} |E|_{8\delta} \geq \left| \bigcup_{p \in \mathcal{P}} \mathcal{F}(p) \right| \geq \delta^{-\gamma(s,t) + \kappa}, \end{displaymath}
provided also that $\delta > 0$ is sufficiently small in terms of $\epsilon,\kappa,s,t$. This implies \eqref{form1}.  \end{proof}

For the proof of the next results, we recall from \cite[Lemma 12.12]{zbMATH01249699} the following Fourier-analytic expression for the $u$-dimensional Riesz-energy:
\begin{equation}\label{form4} I_{u}(\nu) = c_{d,u} \int |\hat{\nu}(\xi)|^{2}|\xi|^{u - d} \, d\xi, \qquad 0 < u < d, \end{equation}
whenever $\nu$ is a finite Borel measure on $\R^{d}$, and where $c_{d,u} > 0$. This expression easily implies the following useful inequality: if $\nu$ is a finite Borel measure on $\R^{d}$, then
\begin{equation}\label{form23} I^{r}_{s}(\nu) \lesssim_{d} I^{\delta}_{s}(\nu), \qquad \delta \leq r \leq 1. \end{equation}
To see this, note, as a first step, that $I^{r}_{s}(\sigma) \leq I_{s}(\sigma)$. This follows from \eqref{form4} applied to $\nu = \sigma \ast \psi_{r}$. As a second step, note that $\psi_{r} \lesssim_{d} \psi_{r} \ast \psi_{\delta}$ for $\delta \leq r \leq 1$, using that $\psi_{r}$ is radially decreasing, and $\spt \psi_{\delta} \subset B(\delta)$. Therefore also $\nu_{r} \lesssim_{d} \nu_{r} \ast \psi_{\delta}$, and finally $I^{r}_{s}(\nu) \lesssim_{d} I^{r}_{s}(\nu_{\delta}) \leq I_{s}(\nu_{\delta}) = I^{\delta}_{s}(\nu)$ by the first step of the proof.

\begin{proposition}\label{prop2} For all $s \in (0,1]$, $t \in [0,2]$, $\kappa \in (0,1]$, and $R > 0$, there exist $\epsilon = \epsilon(\kappa,s,t) > 0$, $k_{0} = k_0(\kappa,s,t) \in \N$, and $\delta_{0} = \delta_{0}(\kappa,s,t,R) > 0$ such that the following holds for all $\delta \in (0,\delta_{0}]$. Assume that $\mu,\sigma$ are Borel probability measures with 
\begin{displaymath} \spt \mu \subset B(R) \quad \text{and} \quad \spt \sigma \subset \mathbb{P}. \end{displaymath}
Assume additionally that $I^{\delta}_{t}(\mu) \leq \delta^{-\epsilon}$ and $I^{\delta}_{s}(\sigma) \leq \delta^{-\epsilon}$. Write $\Pi := \mu \ast \sigma$. Then,
\begin{displaymath} I_{\gamma(s,t)}^{\delta}(\Pi^{k}) \leq \delta^{-\kappa}, \qquad k \geq k_{0}, \end{displaymath} 
where $\gamma(s,t)$ is the constant defined in \eqref{form1}.
\end{proposition}

\begin{proof}
For $r>0$ and $k\in \N$ we write $\Pi^k_r:=\Pi^k\ast \psi_r$ and $J_r(k) := \Vert \Pi_r^{2^{k}}\Vert_2$. Let also $\kappa > 0$ and $\delta \in (0,1]$. Our goal is to show that for every large enough $k$ depending on $s,t$ and $\kappa$, for every $\delta \leq r \leq 1$ we have
    \begin{equation}\label{eq-goal}
J_r(k) \leq \delta^{-\kappa/4} r^{(\gamma(s,t) - 2)/2},
    \end{equation}
    provided that $\delta > 0$ is sufficiently small. From the expression \eqref{form4} together with Plancherel and a dyadic frequency decomposition it then follows that
    $$
I_{\gamma(s,t)}^\delta(\Pi^{2^{k}}) \leq \delta^{-\kappa}.
    $$
    We first observe that the sequence $\{J_r(k)\}_{k\in\N}$ is decreasing in $k$: by Young's inequality (or Plancherel), and using the notation $\|\nu\|_{1}$ for both $L^{1}$-norm and total variance,
    $$
J_r(k+1) = \Vert \Pi^{2^{k}} * \Pi_r^{2^{k}}\Vert_2 \leq \Vert \Pi^{2^{k}}\Vert_1\Vert \Pi_r^{2^{k}}\Vert_2 = \Vert \Pi_r^{2^{k}}\Vert_2 = J_r(k).
    $$
    Therefore, the value of $k$ for which \eqref{eq-goal} holds may depend on $r$ as long as it is uniformly bounded over $r$: choosing the maximum of such $k$ will then work for all $\delta\leq r\leq 1$. 
    
    If $r\geq \delta^{\kappa/8}$, then by the fact that $\gamma(s,t) \leq s+1 \leq 2$, we have
    $$
J_r(k) \leq J_r(0)\leq Cr^{-1} \leq Cr^{-2} \leq \delta^{-\kappa/4} r^{(\gamma(s,t)- 2)/2},
    $$
    and we are done.
    
    Suppose from now on that $\delta\leq r\leq \delta^{\kappa/8}$. Let
    \begin{equation}\label{choiceEpsilon} \epsilon := \frac{\bar{\epsilon}(\kappa/8,s,t)\kappa}{100}, \end{equation}
    where $\bar{\epsilon}(\kappa/8,s,t) > 0$ is the parameter appearing in the statement of Proposition \ref{prop1}. In this argument, the constants in the "$\lesssim$" notation may depend on $s,t,\kappa$, and also $\epsilon$, which is a function of $\kappa,s,t$. 
    
%
     We note that if, for every \( k \leq \lceil 1/\epsilon \rceil \), it holds that  
     \[
     J_r(k) < r^{2\epsilon} J_r(k),
     \]
     then, by applying this inequality \( \lceil 1/\epsilon \rceil \) times, we would have  
     \[
     J_r(1/\epsilon) \leq r^{2\epsilon \cdot 1/\epsilon} J_r(0) \lesssim_{R} r^2 r^{-2} = 1,
     \]
     and \eqref{eq-goal} would hold for sufficiently small \( \delta > 0 \), depending on \( R \).
     
     Thus, we may assume that there exists \( k \leq \lceil 1/\epsilon \rceil \), depending on \( r \), such that  
     \begin{equation}\label{eq-k}
     	J_r(k) \geq J_r(k+1) \geq r^{2\epsilon} J_r(k).
     \end{equation}  
     Fix such a \( k \) satisfying \eqref{eq-k} for the rest of the proof.

    Next we perform a discretisation at scale $r$ of the function $\Pi_r^{2^{k}}$. Let $\mathcal{D}_r(\R^{2})$ be the dyadic squares of side length $r$, and for each $Q \in\mathcal{D}_r(\R^{2})$, set 
    $$
a_Q := \sup_{x\in Q}\Pi_r^{2^{k}}(x).
    $$
   Define $A_j\subseteq \mathbb R^2$ as 
    $$
	{A}_j := \begin{cases}
    \bigcup\lbrace Q \in \mathcal{D}_r(\R^{2}):\ a_{Q} \leq 1\rbrace,\ &j = 0\\
    \bigcup\lbrace Q \in \mathcal{D}_r(\R^{2}) :\ 2^{j-1} < a_{Q} \leq 2^j\rbrace,\ &j\geq 1,
\end{cases}
    $$
 and note that the sets $A_j$ are disjoint for distinct $j$. Since $\|\Pi_r^{2^{k}} \|_\infty\leq r^{-2}$ for all $k\geq 1$, we have $A_j = \emptyset$ for $j\geq 4\log(1/r)+1$. In particular, 
    \begin{equation}\label{eq-discretization1}
\Pi_r^{2^{k}} \leq \sum_{j=0}^{4\log (1/r)+1} 2^j \cdot \mathbf{1}_{A_j}.
    \end{equation}
    In the reverse direction, we have the following relation, justified below:
    \begin{equation}\label{eq-discretization2}
        \sum_{j=0}^{4\log (1/r)+1} 2^j \cdot \mathbf{1}_{A_j} \leq 64\Pi_{8r}^{2^k}.
    \end{equation}
    To see this, let $Q\in\mathcal{D}_r(\R^{2})$ and let $x,x'\in Q$. Now, if $y\in\R^{2}$ is such that $\psi_r(x'-y) > 0$, then by the definition of $\psi_r$ we must have $|x'-y|\leq r$. Since $|x-x'|\leq \diam(Q) \leq 2r$, this implies that $|x-y|/(8r) \leq \bigl(|x-x'|+ |x'-y|\bigr)/(8r) \leq 1/2$ and so $\psi_{8r}(x-y) = (8 r)^{-2} \geq 64^{-1}\psi_r(x'-y)$ by the definition of $\psi_r$. Plugging this into the definition of $\Pi_r^{2^k}$ yields  
    $$
    \Pi_r^{2^k}(x') \leq 64 \Pi_{8r}^{2^k}(x).
    $$
    In particular, $a_{Q} \leq 64\Pi_{8r}^{2^k}(x)$ for any $x \in Q$, and \eqref{eq-discretization2} follows.

    Using \eqref{eq-k}-\eqref{eq-discretization2}, we may pigeonhole a ``large'' set $A$ with the following properties.
    \begin{claim}\label{claim}
    There exists $j\geq 0$ and a set $A:=A_j$ such that
    \begin{itemize}
        \item[(1) \phantomsection \label{1}] $\Vert \Pi_{r}^{2^k}\Vert_2 \lesssim |A|_r^{-\frac{1}{2}} r^{-1-3\epsilon}$ and 
        \item[(2) \phantomsection \label{2}] $\Pi_{8r}^{2^k}(A) \gtrsim r^{3\epsilon}$.
    \end{itemize}
    \end{claim}
    \begin{remark} Claim \ref{claim} is roughly saying that $m = \Pi^{2^{k}}_{r}$ is (almost) a probability measure, and the density of $m$ is a normalised indicator of the form $dm = M\mathbf{1}_{A}$ for some $M > 0$, where $A$ is a union of (dyadic) $r$-squares. Indeed, if $m$ is any measure on $\R^{2}$ with these properties, then $M = |A|^{-1}$ (here $|A|$ is the Lebesgue measure of $A$). Therefore $\|m\|_{2} = M|A|^{1/2} = |A|^{-1/2} = (|A|_{r} \cdot r^{2})^{-1/2} = |A|_{r}^{-1/2}r^{-1}$, as in Claim \ref{claim}(1). 
    
    Claim \ref{claim} only uses \eqref{eq-k}-\eqref{eq-discretization2}, so it shows that any $r$-discretised probability measure (in the sense \eqref{eq-discretization1}-\eqref{eq-discretization2}) whose $L^{2}$-norm does not decrease under convolution (in the sense \eqref{eq-k}) is (roughly) a weighted indicator function (in the sense (1)-(2)).
      \end{remark}

    \begin{proof}[Proof of Claim \ref{claim}] Taking the $L^2$-norms in \eqref{eq-discretization1}, using $\psi_{r} \lesssim \psi_{r} \ast \psi_{r}$, and applying the triangle inequality, we may pigeonhole an index $j\geq 0$ and a set $A = A_j$ such that 
    $$
\Vert \Pi_r^{2^{k+1}}\Vert_2 \lesssim \Vert \Pi_r^{2^{k}} * \Pi_r^{2^k}\Vert_2 \leq (4\log (1/r)+1) \cdot 2^j \Vert \mathbf{1}_A * \Pi_r^{2^k}\Vert_2.
    $$
    On the other hand, by Plancherel and Cauchy-Schwarz,
    $$
\Vert \mathbf{1}_A * \Pi_r^{2^k}\Vert_2 \leq \Vert \mathbf{1}_A * \mathbf{1}_A\Vert_2^{1/2} \Vert \Pi_r^{2^{k+1}}\Vert_2^{1/2}.
    $$
    Combining these two estimates with \eqref{eq-k} yields
    \begin{equation}\label{eq-boundforpi0}
        r^{2\epsilon} \Vert \Pi_r^{2^k}\Vert_2 \leq  \Vert \Pi_r^{2^{k+1}}\Vert_2 \leq (4\log (1/r)+1)^2 2^{2j} \Vert \mathbf{1}_A * \mathbf{1}_A\Vert_2 \lesssim r^{-\epsilon} 2^{2j} \Vert \mathbf{1}_A\Vert_1 \Vert\mathbf{1}_A\Vert_2.
    \end{equation}
   Here, $A$ is a union of elements in $\mathcal{D}_r(\R^{2})$, so 
    $$
\Vert \mathbf{1}_A\Vert_1 = r^2|A|_r \quad  \text{and} \quad \Vert \mathbf{1}_A\Vert_2 = r |A|_r^{1/2}.
    $$
    In particular, combining \eqref{eq-boundforpi0} with $2^j\Vert \mathbf{1}_A\Vert_1 \leq 64\Vert \Pi_{8r}^{2^k}\Vert_1 \leq 64$, we obtain
    \begin{equation}\label{eq-boundforpi}
\Vert \Pi_r^{2^k}\Vert_2 \lesssim r^{-3\epsilon} 2^{2j}\Vert \mathbf{1}_A\Vert_1 \Vert \mathbf{1}_A\Vert_2 \lesssim r^{-3\epsilon} \Vert \mathbf{1}_A\Vert_1^{-1}\Vert \mathbf{1}_A\Vert_2 \leq r^{-3\epsilon-1}|A|_r^{-\frac{1}{2}},
    \end{equation}
     concluding the proof of Part \nref{1} of the claim.
     
    Moving towards proving Part \nref{2}, we first observe that, by \eqref{eq-discretization2}, we have
    \begin{equation}\label{eq-lowerboundfor4r}
2^j \Vert \mathbf{1}_{A}\Vert_2 \leq 64\Vert \Pi_{8r}^{2^{k}}\Vert_2.
    \end{equation}
	The choice of the approximate identity implies that $\psi_{8r}\lesssim \psi_{8r}\ast\psi_{r}$, which in turn yields
    \begin{equation}\label{eq-maximalinequality}
        	\|\Pi^{2^k}_{8r}\|_2 \lesssim \|\Pi^{2^k}\ast \psi_{8r}\ast\psi_{r}\|_2\leq \|\Pi^{2^k}\ast \psi_{r}\|_2 \|\psi_{8r}\|_1=\|\Pi^{2^k}_{r}\|_2.
    \end{equation}
    Finally,
    \begin{displaymath} \|\Pi_{r}^{2^{k}}\|_{2} \stackrel{\eqref{eq-boundforpi}}{\lesssim} r^{-3\epsilon}2^{2j}\|\mathbf{1}_{A}\|_{1}\|\mathbf{1}_{A}\|_{2} \stackrel{\eqref{eq-lowerboundfor4r}}{\lesssim} r^{-3\epsilon}2^{j}\|\mathbf{1}_{A}\|_{1}\|\Pi^{2^{k}}_{8r}\|_{2} \stackrel{\eqref{eq-maximalinequality}}{\lesssim} r^{-3\epsilon}2^{j}\|\mathbf{1}_{A}\|_{1}\|\Pi^{2^{k}}_{r}\|_{2}, \end{displaymath}
    so $r^{3\epsilon} \lesssim 2^j \cdot \Vert \mathbf{1}_A\Vert_1$. Now \eqref{eq-discretization2} allows us to conclude Part \nref{2} of Claim \ref{claim}. \end{proof}
    
   Abbreviate $\bar{\epsilon} := \epsilon(\kappa/8,s,t)$, the constant from Proposition \ref{prop1}, and recall from \eqref{choiceEpsilon} that $\epsilon \leq \bar{\epsilon}/4$. By Part \nref{2} of Claim \ref{claim}, and since $\spt \psi_r \subset B(r)$, 
    \begin{equation}\label{eq-largemeasure}
    \Pi^{2^k}([A]_{8r}) \geq \Pi_{8r}^{2^k}(A) \gtrsim r^{3\epsilon} \geq r^{\bar{\epsilon}}.
    \end{equation}
    Next we wish to apply Proposition \ref{prop1} for $\Pi = \mu*\nu$. This would require the measure $\mu$ to be supported on $B(1)$, but now $\mu$ is supported on $B(R)$. If $R \geq 1$, we will apply the the parabolic scaling $F(x):= (R^{-1}x_1, R^{-2}x_2)$, for $x=(x_1,x_2)\in \R^2$. Note, in particular, that $\spt F_{\sharp}\mu \subset B(1)$ and $\spt F_\sharp\sigma \subset \mathbb{P}$, and  $|F(x) - F(y)| \geq R^{-2}|x - y|$ for $x,y \in \R^{2}$. Since
    $$
    (F_\sharp\mu)_r(F(x))= \frac{1}{r^2}\int \psi\left(\frac{F(x) - F(y)}{r}\right)\,d\mu(y) \leq \frac{R^4}{(R^2r)^2}\int \psi\left(\frac{ x - y}{R^2r}\right)\,d\mu(y) = R^4 \mu_{R^2r}(x)
    $$
    and $R^{-2}\vert x\vert \leq \vert F(x)\vert \leq R^{-1}\vert x \vert$ for every $x\in\R^2$, we have the following bound for the $(r R^{-2})$-discretised energy of $F_\sharp\mu$:
    \begin{align*}
        I_t^{r R^{-2}}(F_\sharp\mu) &= \iint |x-y|^{-t} (F_\sharp\mu)_{R^{-2}r}(x)(F_\sharp\mu)_{R^{-2}r}(y)\,dx\,dy \\
        &\leq R^{2t+8}\iint |F^{-1}(x-y)|^{-t} \mu_r(F^{-1}(x))\mu_r(F^{-1}(y))\,dx\,dy \\
        &= R^{2t+11}\iint |x-y|^{-t} \mu_r(x)\mu_r(y)\,dx\,dy \\
&\leq R^{15} I^r_t(\mu) \stackrel{\eqref{form23}}{\lesssim} R^{15} I^{\delta}_{t}(\mu) \leq R^{15} r^{-8\epsilon/\kappa},
    \end{align*}
    where in the final estimate we used the inequality $I^\delta_t(\mu)\leq \delta^{-\epsilon}$ and the assumption $r\leq \delta^{\kappa/8}$. Similarly, $I_s^{r R^{-2}}(F_\sharp\sigma) \leq R^{15} r^{-8\epsilon/\kappa}$. Since $\epsilon < \bar{\epsilon}\kappa/16$, we have 
    $$
R^{15}r^{-8\epsilon/\kappa} \leq R^{15}r^{-\bar{\epsilon}/2} \leq (r R^{-2})^{-\bar{\epsilon}}
    $$
    provided $R^{15} \leq r^{-\bar{\epsilon}/2}$. Since $r \leq \delta^{\kappa/8}$ in the case we are now considering, this is true for $\delta > 0$ so small that $\delta^{-\bar{\epsilon}\kappa/16} \geq R^{15}$ (this is why the upper bound for "$\delta$" in the statement of Proposition \ref{prop2} is allowed to depend on $R$). For such values of $\delta > 0$, Proposition \ref{prop1} applied for $F_\sharp\Pi = (F_{\sharp}\mu) \ast (F_{\sharp}\sigma)$ (recalling \eqref{eq-largemeasure}) asserts that
    $$
    |A|_{r} \gtrsim R^{-2}|F([A]_{8r}))|_{r R^{-2}} \geq R^{-2}(r R^{-2})^{-\gamma(s,t) + \kappa/8} \geq r^{-\gamma(s,t) + \kappa/4}
    $$
 assuming again that $\delta > 0$, and hence $r$, is small enough depending on $R$. Inserting this into Part \nref{1} of Claim \ref{claim} we obtain
    $$
\Vert \Pi_r^{2^k}\Vert_2 \lesssim r^{(\gamma(s,t) - 2)/2- \kappa/8 - 3\epsilon} \leq \delta^{-\kappa/4}r^{(\gamma(s,t)-2)/2}
    $$
    using $\delta \leq r$ and $\epsilon\leq \kappa/24$. This completes the proof of \eqref{eq-goal}, and Proposition \ref{prop2}. \end{proof}

\begin{cor}\label{cor1} For all $s \in (0,1]$, $t \in [0,\min\{3s,s + 1\})$ and $\kappa > 0$, there exist $\epsilon = \epsilon(\kappa,s,t) > 0$, $\delta_{0} = \delta_{0}(\kappa,s,t) > 0$, and $k_{0} = k_{0}(\kappa,s,t) \in \N$ such that the following holds for all $\delta \in (0,\delta_{0}]$. Assume that $\sigma$ is a Borel probability measure supported on $\mathbb{P}$ with $I_{s}^{\delta}(\sigma) \leq \delta^{-\epsilon}$. Then,
\begin{equation}\label{form6} I_{t}^{\delta}(\sigma^{k}) \leq \delta^{-\kappa}, \qquad k \geq k_{0}. \end{equation} \end{cor}

\begin{proof} The main idea will be to iterate Proposition \ref{prop2} with $\sigma$ fixed, but $\mu$ of the form $\mu = \sigma^{k}$ for increasing values of $k \in \N$. We will often implicitly use the fact that the sequence of Riesz-energies $\{I_{u}(\nu^{k})\}_{k \in \N}$ is decreasing in $k$, whenever $\nu$ is a probability measure on $\R^{2}$. This is immediate from \eqref{form4}, and since
\begin{displaymath} |\widehat{\nu^{k + 1}}(\xi)| = |\widehat{\nu^{k}}(\xi)||\hat{\nu}(\xi)| \leq |\widehat{\nu^{k}}(\xi)|, \qquad \xi \in \R^{2}. \end{displaymath}

Consider the sequence of exponents $\{t_{j}\}_{j = 0}^{\infty} \subset [s,\min\{3s,s + 1\}]$, where
\begin{displaymath} t_{0} := s \quad \text{and} \quad t_{j + 1} := \gamma(s,t_{j}) := \min\left\{s + t_{j},\tfrac{3s + t_{j}}{2},s + 1\right\}. \end{displaymath}
It is easy to check that for $t \in [0,\min\{3s,s + 1\})$ given, we have $t_{j} \geq t$ for all $j \geq j_{0}(s,t)$ (one can read off from \eqref{form12} that $2^{-j + 2} \leq (3s - t)/s$ suffices). Therefore, Corollary \ref{cor1} will follow once we manage to show the following: for $j \geq 0$ and $\kappa_{j} > 0$ fixed,
\begin{equation}\label{form7} I_{t_{j}}^{\delta}(\sigma^{k}) \leq \delta^{-\kappa_{j}}, \qquad k \geq k(j,\kappa_{j},s), \end{equation}
provided that $\epsilon = \epsilon(\kappa_{j},j,s) > 0$ is sufficiently small in the hypothesis $I_{s}(\sigma) \leq \delta^{-\epsilon}$, and $\delta > 0$ is small enough in terms of $s,j,\kappa_{j},\epsilon$. Once this has been verified, \eqref{form6} follows by taking $j := j_{0}(s,t)$ and $\kappa_{j} := \kappa$.

We prove \eqref{form7} by induction on "$j$". The case $j = 0$ is clear, because $t_{0} = s$: we can simply take $\epsilon := \kappa_{0}$ and $k(0,\kappa_{0},s) := 1$. Let us then assume that \eqref{form7} has already been established for a certain $j \geq 0$, and for all $\kappa_{j} > 0$.

We then consider the case "$j + 1$". Fix $\kappa_{j + 1} > 0$, and apply Proposition \ref{prop2} with the parameters $s,t_{j},\kappa_{j + 1}$, and $R > 0$ to be determined in a moment. This provides us with another parameter $\bar{\epsilon} = \epsilon(\kappa_{j + 1},s,t_{j}) > 0$ such that the following holds. Assume that $I_s^\delta(\sigma) \leq \delta^{-\bar{\epsilon}}$ and $\mu$ is a probability measure on $B(R)$ with $I_{t_{j}}^{\delta}(\mu) \leq \delta^{-\bar{\epsilon}}$, and write $\Pi := \mu \ast \sigma$. Then, if $\delta \leq \delta_{0}(s,t_{j},\kappa_{j + 1},R)$, and $k \geq k(\kappa_{j + 1},s,t_{j})$, 
\begin{equation}\label{form8} I_{t_{j + 1}}^{\delta}(\Pi^{k}) = I^{\delta}_{\gamma(s,t_{j})}(\Pi^{k}) \leq \delta^{-\kappa_{j + 1}}. \end{equation}

Now, we apply the inductive hypothesis \eqref{form7} at index "$j$" with the choice 
\begin{displaymath} \kappa_{j} := \bar{\epsilon} = \epsilon(\kappa_{j + 1},s,t_{j}) > 0. \end{displaymath}
The conclusion is that if $I^{\delta}_{s}(\sigma) \leq \delta^{-\epsilon}$ for $\epsilon = \epsilon(\kappa_{j},j,s) > 0$ small enough, and if $k \geq k(j,\kappa_{j},s)$, then in fact $I^{\delta}_{t_{j}}(\sigma^{k}) \leq \delta^{-\kappa_{j}} = \delta^{-\bar{\epsilon}}$. Thus, Proposition \ref{prop2} is applicable to the measure $\mu := \sigma^{k}$ for any $k \geq k(j,\kappa_{j},s)$. For concreteness, we define
\begin{displaymath} \mu := \sigma^{k(j,\kappa_{j},s)}, \end{displaymath}
and $\Pi := \mu \ast \sigma = \sigma^{k(j,\kappa_{j},s) + 1}$. Notice that 
\begin{displaymath} \spt \mu \subset B(2^{k(j,\kappa_{j},s)}) =: B(R). \end{displaymath}
This determines the choice of $R$ (there is no circular reasoning, because the choice of $R$ does not affect the exponent $\bar{\epsilon} = \epsilon(\kappa_{j + 1},s,t_{j})$). We may deduce from \eqref{form8} that if $\delta \leq \delta_{0}(s,t_{j},\kappa_{j + 1},R)$ (all of these constants eventually only depend on $\kappa_{j + 1},j,s$, as required), then
\begin{displaymath} I^{\delta}_{t_{j + 1}}(\sigma^{[k(j,\kappa_{j},s) + 1] \cdot k}) = I_{t_{j + 1}}^{\delta}(\Pi^{k}) \leq \delta^{-\kappa_{j + 1}}, \qquad k \geq k(\kappa_{j + 1},s,t_{j}). \end{displaymath}
In particular, provided that $I_{s}^{\delta}(\sigma) \leq \delta^{-\epsilon}$, 
\begin{displaymath} I_{t_{j + 1}}^{\delta}(\sigma^{k}) \leq \delta^{-\kappa_{j + 1}}, \qquad k \geq [k(j,\kappa_{j},s) + 1] \cdot k(\kappa_{j + 1},s,t_{j}). \end{displaymath}
This verifies \eqref{form7} for index "$j + 1$", and completes the proof of the corollary. \end{proof}

\section{Proofs of Theorem \ref{main} and Corollary \ref{cor2}}\label{s3}

We are then prepared to prove Theorem \ref{main}.

\begin{proof}[Proof of Theorem \ref{main}] Fix $s \in [0,1]$, $t \in [0,\min\{3s,s + 1\})$, and a Borel measure on $\mathbb{P}$ satisfying $\sigma(B(x,r)) \leq r^{s}$ for all $x \in \R^{2}$ and $r > 0$. This implies, in particular, that $I^{\delta}_{s}(\sigma) \lesssim \log(1/\delta)$ for all $\delta \in (0,1]$. Our claim is that if $p \geq 1$ is sufficiently large, depending only on $s,t$, then
\begin{equation}\label{form9} \|\hat{\sigma}\|_{L^{p}(B(R))} \leq C_{s,t}R^{(2 -t)/p}, \qquad R \geq 1. \end{equation}
This is clear if $s = 0$ (since $\|\hat{\sigma}\|_{\infty} \leq \|\sigma\| \lesssim 1$), so we may assume $s \in (0,1]$.

Fix $R \geq 1$ and $t < u < \min\{3s,s + 1\}$, and let $\{\varphi_{\delta}\}_{\delta > 0}$ be an approximate identity with the property $|\widehat{\varphi_{\delta}}(\xi)| \gtrsim 1$ for $|\xi| \leq \delta^{-1}$. Then, writing $\delta := R^{-1}$, and for $p = 2k \in 2\N$,
\begin{align*} \|\hat{\sigma}\|_{L^{p}(B(R))}^{p} & = \int_{B(R)} |\hat{\sigma}(\xi)|^{p} \, d\xi = \int_{B(R)} |\widehat{(\sigma^{k})}(\xi)|^{2} \, d\xi\\
& \lesssim \int_{B(R)} |\widehat{(\sigma^{k})_{\delta}}(\xi)|^{2} \, d\xi \leq R^{2 - u} \int |\widehat{(\sigma^{k})_{\delta}}(\xi)|^{2}|\xi|^{u - 2} \, d\xi \sim_{u} R^{2 - u}I^{\delta}_{u}(\sigma^{k}). \end{align*}
For concreteness, let us pick $u := \tfrac{1}{2}(t + \min\{3s,s + 1\})$, and write $\kappa := u - t > 0$ (thus both $\kappa,s$ are functions of $s,t$). Now, according to \eqref{form6}, we have $I_{u}^{\delta}(\sigma^{k}) \leq \delta^{-\kappa} = R^{u - t}$, provided that $k \geq k_{0}(s,t)$, and $\delta \leq \delta_{0}(s,t)$; the hypothesis $I_{s}^{\delta}(\sigma) \leq \delta^{-\epsilon}$ is automatically satisfied for $\delta > 0$ small enough, since $I_{s}^{\delta}(\sigma) \lesssim \log(1/\delta)$. For such $k$, and $p = 2k$, we have now established that
\begin{displaymath} \|\hat{\sigma}\|_{L^{p}(B(R))}^{p} \lesssim_{s,t} R^{2 - u}I_{u}^{\delta}(\sigma^{k}) \leq R^{2 - t}, \qquad R \geq 1/\delta_{0}. \end{displaymath}
This formally implies \eqref{form6}, since $\delta_{0} > 0$ only depends on $s,t$. \end{proof}

Finally, we record how Corollary \ref{cor2} follows by combining Theorem \ref{main} with the following simple relation between the ($\delta$-discrete) additive energy of a finite set and the ($\delta$-discrete) energy of the counting measure on it.
\begin{proposition}\label{prop52}
    Let $P\subseteq \R^2$ be a finite set, and let $\sigma := |P|^{-1}\mathcal{H}^{0}|_{P}$ denote the normalised counting measure on $P$. For any $t,\delta>0$ and $n\in\N$ such that $I_t^\delta(\sigma^n) \lesssim 1$, it holds that 
    \begin{displaymath}
|\{(p_{1},\ldots,p_{n},q_{1},\ldots,q_{n}) \in P^{2n} : |(p_{1} + \ldots + p_{n}) - (q_{1} + \ldots + q_{n})| \leq \delta\}| \lesssim \delta^t |P|^{2n}.
   \end{displaymath}
\end{proposition}

\begin{proof}[Proof of Proposition \ref{prop52}] Note that 
\begin{displaymath} \frac{|\{(p_{1},\ldots,p_{n},q_{1},\ldots,q_{n}) \in P^{2n} : |(p_{1} + \ldots + p_{n}) - (q_{1} + \ldots + q_{n})| \leq \delta\}|}{|P|^{2n}} = \int \sigma^{n}(B(z,\delta)) \, d\sigma^{n}(z). \end{displaymath}
Let $\{\psi_{\delta}\}_{\delta > 0} \subset C^{\infty}(\R^{2})$ be an approximate identity of the standard form $\psi_{\delta}(z) = \delta^{-2}\psi(z/\delta)$, where $\mathbf{1}_{B(1/2)} \leq \psi \leq \mathbf{1}_{B(1)}$. Writing $\sigma^{n}_{\delta} := (\sigma^{n}) \ast \psi_{\delta}$, it is then easy to check that
\begin{displaymath} \sigma^{n}(B(z,\delta)) \lesssim \delta^{2} \inf_{x \in B(z,\delta)} \sigma^{n}_{4\delta}(x), \end{displaymath}
and in particular $\sigma^{n}(B(z,\delta)) \lesssim \int_{B(z,\delta)} \sigma^{n}_{4\delta}(x) \, dx$. Therefore,
\begin{align*} \int \sigma^{n}(B(z,\delta)) \, d\sigma^{n}(z) & \lesssim \int \int_{B(z,\delta)} \sigma_{4\delta}^{n}(x) \, dx \, d\sigma^{n}(z)\\
& = \int \sigma_{4\delta}^{n}(x) \sigma^{n}(B(x,\delta)) \, dx \lesssim \delta^{2} \|\sigma_{4\delta}^{n}\|_{2}^{2}. \end{align*}
As explained in Remark \ref{rem2}, the condition $I_t^\delta(\sigma^n) \lesssim 1$ implies that $\delta^{2}\|\sigma^{n}_{4\delta}\|_{2}^{2} \lesssim \delta^{t}$. This completes the proof. 
\end{proof}

\begin{proof}[Proof of Corollary \ref{cor2}]
Let $\sigma := |P|^{-1}\mathcal{H}^{0}|_{P}$ denote the normalised counting measure on $P$. Using the $(\delta,s)$-set condition of $P$, the measure $\sigma$ satisfies $\sigma(B(x,r)) \lesssim r^{s}$ for all $\delta \leq r \leq 1$. This is good enough to apply Theorem \ref{main} at scale $R = \delta^{-1}$ (it would be possible to apply Theorem \ref{main} literally to the measure $\sigma_{\delta} = \sigma \ast \psi_{\delta}$, but in fact the proof of Theorem \ref{main} with given $R \geq 1$ only uses the non-concentration condition of $\sigma$ on scales $[R^{-1},1]$). We infer from the equivalent formulation in Remark \ref{rem2} that $I^\delta_t(\sigma^n) \lesssim_{s,t} 1$, provided that $n \geq 1$ is sufficiently large, depending on $s,t$. An application of Proposition \ref{prop52} completes the proof of Corollary \ref{cor2}. 
\end{proof}

\section{Proof of Proposition \ref{mainProp}}\label{s5}

\begin{thm}\label{thm2} Let $s \in (\tfrac{1}{2},1]$ and $t \in (0,2)$. Let $\mu$ be a Borel measure on $B(1)$ with $I_{t}(\mu) \leq 1$, and let $\sigma$ be a Borel measure on $\mathbb{P}$ satisfying $\sigma(B(x,r)) \leq r^{s}$ for all $x \in \R^{2}$ and $r > 0$. Then, 
\begin{equation}\label{form30} \|(\mu \ast \sigma)_{\delta}\|_{L^{2}(\mathbb R^2)}^{2} \lesssim_{s,t,\epsilon} \delta^{\zeta(s,t) - 2 - \epsilon}, \qquad \epsilon > 0, \end{equation}
where $\zeta(s,t) = \min\{t + (2s - 1),s + 1\}$. Consequently, $I_{\zeta(s,t) - \epsilon}(\mu \ast \sigma) \lesssim_{s,t,\epsilon} 1$ for $\epsilon > 0$. \end{thm}

\begin{remark} Note that in proving Theorem \ref{thm2}, we may assume $s + t \leq 2$. Let us see how to reduce the case $s + t > 2$ to the case $s + t \leq 2$. Assume $s + t > 2$. Then $\zeta(s,t) = s + 1$. Set $\underline{t} := 2 - s < t$, and note that $I_{\underline{t}}(\mu) \lesssim I_{t}(\mu) \leq 1$. Now $s + \underline{t} \leq 2$, so taking that case for granted, we may deduce \eqref{form30} with exponent $\zeta(s,\underline{t})$. However, $\zeta(s,\underline{t}) = s + 1 = \zeta(s,t)$.

A second remark is that the proof of Theorem \ref{thm2} does not explicitly use that $s > \tfrac{1}{2}$, but \eqref{form30} is trivial (that is: follows from $I_{t}(\mu) \leq 1$) if $s \leq \tfrac{1}{2}$. \end{remark}

 Theorem \ref{thm2} will be deduced from a $\delta$-discretised statement, Proposition \ref{thm3}. To pose the statement rigorously, we introduce the following terminology:
\begin{definition}[$\delta$-measures] A collection of non-negative weights $\boldsymbol{\mu} = \{\mu(p)\}_{p \in \mathcal{D}_{\delta}}$ with $\|\boldsymbol{\mu}\| := \sum_{p} \mu(p) \leq 1$ is called a \emph{$\delta$-measure}. We say that a $\delta$-measure $\boldsymbol{\mu}$ is a $(\delta,s,C)$-measure if $\boldsymbol{\mu}(Q) \leq C\ell(Q)^{s}$ for all $Q \in \mathcal{D}_{r}$ with $\delta \leq r \leq 1$. We also define the $s$-dimensional energy
\begin{displaymath} I_{s}(\boldsymbol{\mu}) := 1 + \sum_{p \neq q} \frac{\mu(p)\mu(q)}{\dist(p,q)^{s}}, \end{displaymath}
where $\dist(p,q)$ is the distance between the midpoints of $p,q$. \end{definition}

\begin{proposition}[Discretised version of Theorem \ref{thm2}]\label{thm3} Let $s \in [0,1]$ and $t \in (0,2)$ such that $s + t \leq 2$. Let $\boldsymbol{\mu}$ be a $\delta$-measure. For every $p \in \mathcal{D}_{\delta}$, let $\boldsymbol{\sigma}_{p}$ be a $(\delta,s,1)$-measure supported on $p + [\mathbb{P}]_\delta := \{q \in \mathcal{D}_{\delta} : q \cap (p + [\mathbb{P}]_\delta) \neq \emptyset\}$. Then,
\begin{equation}\label{form29} \int \Big( \sum_{p \in \mathcal{D}_{\delta}} \mu(p) \sum_{q \in \mathcal{D}_{\delta}} \sigma_{p}(q) \cdot (\delta^{-2}\mathbf{1}_{q}) \Big)^{2} \, dx \lesssim_{\epsilon} I_{t}(\boldsymbol{\mu})\delta^{2s + t - 3 - \epsilon}, \qquad \epsilon > 0. \end{equation}
\end{proposition}

The reduction from Proposition \ref{thm3} to Theorem \ref{thm2} is standard; we sketch the details.
\begin{proof}[Proof of Theorem \ref{thm2} assuming Proposition \ref{thm3}] Start by observing that $(\mu \ast \sigma)_{\delta} = (\mu \ast \sigma) \ast \psi_{\delta} \lesssim (\mu \ast \sigma) \ast (\psi_{\delta} \ast \psi_{\delta}) = \mu_{\delta} \ast \sigma_{\delta}$. We bound both measures $\mu_{\delta}$ and $\sigma_{\delta}$ point-wise:
\begin{displaymath} \mu_{\delta} \lesssim \sum_{p \in \mathcal{D}_{\delta}} \bar{\mu}(p) \cdot (\delta^{-2}\mathbf{1}_{p}) \quad \text{and} \quad \sigma_{\delta} \lesssim \sum_{q \in \mathcal{D}_{\delta}} \bar{\sigma}(q) \cdot (\delta^{-2}\mathbf{1}_{q}).  \end{displaymath}  
Here $\bar{\boldsymbol{\mu}} = \{\bar{\mu}(p)\}_{p \in \mathcal{D}_{\delta}}$ is a suitable $\delta$-measure satisfying $I_{t}(\bar{\boldsymbol{\mu}}) \lesssim \max\{I_{t}(\mu),1\}$, and $\bar{\boldsymbol{\sigma}} = \{\bar{\sigma}(q)\}_{q \in \mathcal{D}_{\delta}}$ is a suitable $(\delta,s,O(1))$-measure supported on the $\delta$-neighbourhood of $\mathbb{P}$, denoted by $[\mathbb{P}]_\delta$ as usual. By the linearity of convolution, and using $\delta^{-2}(1_{p} \ast 1_{q}) \leq \mathbf{1}_{p + q}$, we find
\begin{displaymath} (\mu_{\delta} \ast \sigma_{\delta})(x) \lesssim \sum_{p \in \mathcal{D}_{\delta}} \bar{\mu}(p) \sum_{q \in \mathcal{D}_{\delta}} \bar{\sigma}(q) (\delta^{-2}\mathbf{1}_{p + q})(x). \end{displaymath} 
Now, for $p \in \mathcal{D}_{\delta}$ fixed, we define the weights $\sigma_{p}(p + q) := \bar{\sigma}(q)$. (We leave to the reader the fine print that $p + q$ is not exactly a dyadic square.) The ensuing $\delta$-measure $\boldsymbol{\sigma}_{p} = \{\sigma_{p}(q)\}_{q \in \mathcal{D}_{\delta}}$ is a $(\delta,s,O(1))$-measure supported on $p + [\mathbb{P}]_\delta$, and we have just seen that the convolution $(\mu \ast \sigma)_{\delta}$ is point-wise bounded by the expression
\begin{displaymath} \sum_{p \in \mathcal{D}_{\delta}} \bar{\mu}(p) \sum_{q \in \mathcal{D}_{\delta}} \sigma_{p}(q)(\delta^{-2}\mathbf{1}_{q}). \end{displaymath}
Therefore, \eqref{form29} implies \eqref{form30}. \end{proof}

We will further reduce the proof of Proposition \ref{thm3} to the following version, where the "energy" assumption $I_{t}(\boldsymbol{\mu}) \leq 1$ is replaced by a Frostman condition. 
\begin{proposition}\label{thm4} Let $s \in [0,1]$ and $t \in (0,2)$ such that $s + t \leq 2$, and let $\mathbf{C}_{\mu},\mathbf{C}_{\sigma} \geq 1$. Let $\boldsymbol{\mu}$ be a $(\delta,t,\mathbf{C}_{\mu})$-measure. For each $p \in \mathcal{D}_{\delta}$, let $\boldsymbol{\sigma}_{p}$ be a $(\delta,s,\mathbf{C}_{\sigma})$-measure supported on $p + [\mathbb{P}]_\delta := \{q \in \mathcal{D}_{\delta} : q \cap (p + [\mathbb{P}]_\delta) \neq \emptyset\}$. Then,
\begin{equation}\label{form32} \int \Big( \sum_{p \in \mathcal{D}_{\delta}} \mu(p) \sum_{q \in \mathcal{D}_{\delta}} \sigma_{p}(q) \cdot (\delta^{-2}\mathbf{1}_{q}) \Big)^{2} \, dx \lesssim_{\epsilon} (\mathbf{C}_{\mu}\mathbf{C}_{\sigma}^{2}\|\boldsymbol{\mu}\|) \cdot \delta^{2s + t - 3 - \epsilon} + 1, \qquad \epsilon > 0. \end{equation} \end{proposition}

\begin{proof}[Proof of Proposition \ref{thm3} assuming Proposition \ref{thm4}] Let $\boldsymbol{\mu}$ and $\boldsymbol{\sigma}_{p}$ be the objects from Proposition \ref{thm3}, so in particular $I_{t}(\boldsymbol{\mu}) \leq 1$. For each $j \geq 1$, let $\boldsymbol{\mu}_{j}$ be the $\delta$-measure obtained by restricting $\boldsymbol{\mu}$ to the squares $p \in \mathcal{D}_{\delta}$ satisfying
\begin{equation}\label{form31} 2^{j - 1} \leq \sum_{q : q \neq p} \frac{\mu(q)}{\dist(p,q)^{t}} \leq 2^{j}. \end{equation}
For $j = 0$, we define $\boldsymbol{\mu}_{j} = \boldsymbol{\mu}_{0}$ similarly, except that the lower bound "$2^{j - 1} \leq$" is omitted. It is now easy to check, using $I_{t}(\boldsymbol{\mu}) \leq 1$, that each $\boldsymbol{\mu}_{j}$ is a $(\delta,t,O(2^{j}))$-measure with $\|\boldsymbol{\mu}_{j}\| \lesssim I_{t}(\boldsymbol{\mu})2^{-j}$. Another easy observation is that $\boldsymbol{\mu}_{j} \equiv 0$ for $2^{j} > \delta^{-2}$, since the expression in \eqref{form31} is bounded from above by $\|\boldsymbol{\mu}\| \cdot \delta^{-2} \leq \delta^{-2}$.

Now, fixing $x \in \R^{2}$, note the point-wise inequality
 \begin{displaymath} \sum_{p \in \mathcal{D}_{\delta}} \mu(p) \sum_{q \in \mathcal{D}_{\delta}} \sigma_{p}(q) \cdot (\delta^{-2}\mathbf{1}_{q}(x)) \lesssim (\log(1/\delta))^{2} \max_{j} \sum_{p \in \mathcal{D}_{\delta}} \mu_{j}(p) \sum_{q \in \mathcal{D}_{\delta}} \sigma_{p}(q)\cdot (\delta^{-2}\mathbf{1}_{q}(x)). \end{displaymath}
 Consequently, the left hand side of \eqref{form32} is bounded from above by
 \begin{displaymath} (\log(1/\delta))^{O(1)}\sum_{j = 1}^{(\log 1/\delta)^{2}} \int \Big( \sum_{p \in \mathcal{D}_{\delta}} \mu_{j}(p) \sum_{q \in \mathcal{D}_{\delta}} \sigma_{p}(q) \cdot (\delta^{-2}\mathbf{1}_{q}) \Big)^{2} \, dx. \end{displaymath}
 Since each of the $\delta$-measures $\boldsymbol{\mu}_{j}$ satisfy the Frostman condition suited for an application of Proposition \ref{thm4}, the expression above is bounded from above by
 \begin{displaymath} \lesssim (\log(1/\delta))^{O(1)}\sum_{j = 1}^{(\log 1/\delta)^{2}} \left[ (2^{j}\|\boldsymbol{\mu}_{j}\|) \cdot \delta^{2s + t - 3 - \epsilon} + 1 \right] \lesssim_{\epsilon} I_{t}(\boldsymbol{\mu})\delta^{2s + t - 3 - 2\epsilon}. \end{displaymath} 
 This completes the proof of Proposition \ref{thm3}. \end{proof}

It then remains to prove Proposition \ref{thm4}. We first recall the notion of Katz-Tao $(\delta,s)$-set.
\begin{definition}[Katz-Tao $(\delta,s)$-set]
	Let $(X,d)$ be a metric space.  We say that a $\delta$-separated set $P\subset X$ is a \emph{Katz-Tao $(\delta, s, C)$-set} if
	\[
	|P\cap B(x,r)|\leq C\Bigl(\frac{r}{\delta}\Bigr)^s, \qquad x\in X, r\geq \delta.
	\]
\end{definition}
The definition extends to families of lines and $\delta$-tubes in the same way as the notion of $(\delta,s)$-sets in Section \ref{s:prelim}. A slight variant of the notion above was initially introduced by Katz and Tao \cite{KT01}. The main tool in the proof of Proposition \ref{thm4} will be the following incidence bound of Fu and Ren \cite[Theorem 5.2]{fu2022incidence}:
\begin{thm}\label{t:fuRen} Let $s \in [0,1]$ and $t \in [0,2]$ such that $s + t \leq 2$, and let $C_{1},C_{2} \geq 1$. Let $\mathcal{T}$ be a Katz-Tao $(\delta,t,C_{1})$-set of $\delta$-tubes. For every $T \in \mathcal{T}$, let $\mathcal{F}(T) \subset \{p \in \mathcal{D}_{\delta} : p \cap T \neq \emptyset\}$ be a Katz-Tao $(\delta,s,C_{2})$-set. Then, with $\mathcal{F} = \bigcup_{T \in \mathcal{T}} \mathcal{F}(T)$,
\begin{displaymath} \sum_{T \in \mathcal{T}} |\mathcal{F}(T)| \lesssim_{\epsilon} \delta^{-\epsilon}\sqrt{\delta^{-1}C_{1}C_{2}|\mathcal{F}||\mathcal{T}|}, \qquad \epsilon > 0. \end{displaymath} \end{thm}

\begin{remark} Theorem \ref{t:fuRen} was used in \cite{fu2022incidence} to solve the cases $s + t \geq 2$ of the Furstenberg set conjecture. As already mentioned at the beginning of Section \ref{s2}, the conjecture was fully settled in 2023 by Ren and Wang \cite{2023arXiv230808819R}. \end{remark}

In fact, we will need a version of Theorem \ref{t:fuRen}, where $\delta$-tubes are replaced by translates of $[\mathbb{P}]_\delta$. The following version can be deduced from Theorem \ref{t:fuRen} by means of the map $\Psi(x,y) = (x,x^{2} - y)$ introduced above Proposition \ref{prop3}. We omit the details, because the argument is so similar to the deduction of Lemma \ref{lemma1} from Theorem \ref{thm:renWang}.

\begin{thm}\label{t:fuRenParabolas} Let $s \in [0,1]$ and $t \in [0,2]$ such that $s + t \leq 2$, and let $C_{1},C_{2} \geq 1$. Let $\mathcal{P} \subset \mathcal{D}_{\delta}$ be a Katz-Tao $(\delta,t,C_{1})$-set. For every $p \in \mathcal{P}$, let $\mathcal{F}(p) \subset \{q \in \mathcal{D}_{\delta} : q \cap (p + [\mathbb{P}]_\delta) \neq \emptyset\}$ be a Katz-Tao $(\delta,s,C_{2})$-set. Then, with $\mathcal{F} = \bigcup_{p \in \mathcal{P}} \mathcal{F}(p)$,
\begin{equation}\label{fuRen} \sum_{p \in \mathcal{P}} |\mathcal{F}(p)| \lesssim_{\epsilon} \delta^{-\epsilon}\sqrt{\delta^{-1}C_{1}C_{2}|\mathcal{F}||\mathcal{P}|}, \qquad \epsilon > 0. \end{equation} \end{thm}

\begin{proof}[Proof of Proposition \ref{thm4}] We start by a pigeonholing exercise, reducing to a special case of Proposition \ref{thm4} where all the coefficients $\mu(p)$ and $\sigma_{p}(q)$ are constant, up to a factor of $2$. First, let $\mathcal{D}_{\delta}(j,i) := \{(p,q) \in \mathcal{D}_{\delta}\times\mathcal{D}_\delta : 2^{-j - 1} \leq \mu(p) \leq 2^{-j}, 2^{-i-1}\leq \sigma_p(q) \leq 2^{-i}\}$ for $2^{-j-i}\geq \delta^{6}$, and $\mathcal{D}_{\delta}(\infty) := \{(p,q) : \mu(p)\sigma_p(q) < \delta^6\}$. With this notation, we have the point-wise bound
\begin{displaymath} \sum_{p \in \mathcal{D}_{\delta}} \mu(p) \sum_{q \in \mathcal{D}_{\delta}} \sigma_{p}(q) \cdot (\delta^{-2}\mathbf{1}_{q}) \lessapprox \max_{(j,i)}  \sum_{(p,q) \in \mathcal{D}_{\delta}(j,i)} \mu(p) \sigma_{p}(q) \cdot (\delta^{-2}\mathbf{1}_{q}), \end{displaymath}
so
\begin{displaymath} \int \Big( \sum_{p \in \mathcal{D}_{\delta}} \mu(p) \sum_{q \in \mathcal{D}_{\delta}} \sigma_{p}(q) \cdot (\delta^{-2}\mathbf{1}_{q}) \Big)^{2} \, dx \lessapprox \sum_{j,i} \int \Big( \sum_{(p,q) \in \mathcal{D}_{\delta}(j,i)} \mu(p) \sigma_{p}(q) \cdot (\delta^{-2}\mathbf{1}_{q}) \Big)^{2} \, dx. \end{displaymath}
Here the "$\lessapprox$" notation hides absolute powers of $\log(1/\delta)$. Since there are $\lessapprox 1$ terms in the sum, it suffices to prove \eqref{form32} individually for each term. In the case $2^{-j-i} \geq \delta^{6}$, the values of $\mu(p)$ and $\sigma_p(q)$ are almost constant, as desired. Moreover, the $(j,i) = \infty$ term is estimated using trivial bounds as
\begin{displaymath} \leq \delta^{12}\int_{[0,1]^{2}} \left( |\mathcal{D}_{\delta}|^{2}\delta^{-2} \right)^{2} \, dx \leq 1.  \end{displaymath} 
Now, fixing $(j,i)$ with $2^{-j-i} \geq \delta^{6}$, we have found:
\begin{itemize}
\item[(a)] A subset $\mathcal{P} \subset \mathcal{D}_{\delta}$ and a number $\mu := 2^{-j} \geq \delta^{6}$ such that $\mu(p) \sim \mu$ for all $p \in \mathcal{P}$. 
\item[(b)] For every $p \in \mathcal{P}$ a subset $\mathcal{F}(p) \subset \mathcal{D}_{\delta}$ and a number $\sigma := 2^{-i} \geq \delta^{6}$ such that $\sigma_{p}(q) \sim \sigma$ for all $q \in \mathcal{F}(p)$. In particular, since $\sigma_{p}$ is supported on $p + [\mathbb{P}]_\delta$, 
\begin{displaymath} \cup \mathcal{F}(p) \subset p + [\mathbb{P}]_\delta. \end{displaymath} 
\end{itemize} 
To complete the proof of \eqref{form32}, it suffices to show
\begin{equation}\label{form33} (\mu \sigma \delta^{-2})^{2} \int \Big( \sum_{p \in \mathcal{P}} \mathbf{1}_{\mathcal{F}(p)} \Big)^{2} \, dx \lesssim_{\epsilon} (\mathbf{C}_{\mu}\mathbf{C}_{\sigma}^{2}\|\boldsymbol{\mu}\|) \cdot \delta^{2s + t - 3 - \epsilon}, \qquad \epsilon > 0. \end{equation}
We start by recording the following facts about the non-concentration of $\mathcal{P}$ and each $\mathcal{F}(p)$:
\begin{itemize}
\item $\mathcal{P}$ is a Katz-Tao $(\delta,t,O(\mathbf{C}_{\mu}\delta^{t}\mu^{-1}))$-set. Note that $\mathbf{C}\delta^{t}\mu^{-1} \gtrsim 1$.
\item $\mathcal{F}(p)$ is a Katz-Tao $(\delta,s,O(\mathbf{C}_{\sigma}\delta^{s}\sigma^{-1}))$-set for each $p \in \mathcal{P}$. Note that $\delta^{s}\sigma^{-1} \gtrsim 1$.
\end{itemize}
The proofs are similar, let us check this for $\mathcal{P}$. Fix $Q \in \mathcal{D}_{r}$ with $\delta \leq r \leq 1$. Then,
\begin{displaymath} \mu \cdot |\mathcal{P} \cap Q| \stackrel{\textup{(a)}}{\lesssim} \mu(Q) \leq \mathbf{C}_{\mu}r^{t} \quad \Longrightarrow \quad |\mathcal{P} \cap Q| \leq (\mathbf{C}_{\mu}\delta^{t}\mu^{-1}) \cdot \left(\tfrac{r}{\delta} \right)^{t}. \end{displaymath}
Now, to proceed estimating \eqref{form33}, fix $r \in 2^{\N} \cup \{1\}$, and consider the "$r$-rich squares"
\begin{displaymath} \mathcal{F}_{r} := \{q \in \mathcal{D}_{\delta} : r_{\mathcal{P}}(p) \in [r,2r]\}, \quad \text{where} \quad r_{\mathcal{P}}(q) := \sum_{p \in \mathcal{P}} \mathbf{1}_{\mathcal{F}(p)}(q). \end{displaymath}
With this notation, and noting that $\|r_{\mathcal{P}}\|_{L^{\infty}} \leq |\mathcal{P}| \leq \delta^{-2}$,
\begin{equation}\label{form35} \int \Big( \sum_{p \in \mathcal{P}} \mathbf{1}_{\mathcal{F}(p)} \Big)^{2} \, dx \lesssim \sum_{r = 1}^{\delta^{-2}} r^{2} \cdot \mathrm{Leb}(\cup \mathcal{F}_{r}). \end{equation}
We then plan to estimate the Lebesgue measures of the $r$-rich sets $\mathcal{F}(r)$ separately, and indeed we claim that
\begin{equation}\label{form34} \mathrm{Leb}( \mathcal{F}_{r}) = \delta^{2}|\mathcal{F}_{r}| \lesssim_{\epsilon} \frac{\mathbf{C}_{\mu}\mathbf{C}_{\sigma}(\mu \sigma)^{-1}\delta^{s + t + 1 - \epsilon}|\mathcal{P}|}{r^{2}}, \qquad r \geq 1. \end{equation}
Before proving \eqref{form34}, let us check that it implies \eqref{form33}, and therefore \eqref{form32}. Indeed, plugging \eqref{form34} into \eqref{form35}, we find
\begin{align*} (\mu \sigma \delta^{-2})^{2} \int \Big(\sum_{p \in \mathcal{P}} \mathbf{1}_{\mathcal{F}(p)} \Big)^{2} \, dx & \lesssim_{\epsilon} \mathbf{C}_{\mu}\mathbf{C}_{\sigma}(\mu \sigma)\delta^{s + t - 3 - \epsilon}|\mathcal{P}|\\
& \lesssim \mathbf{C}_{\mu}\mathbf{C}_{\sigma}^{2}\mu \delta^{2s + t - 3 - \epsilon}|\mathcal{P}| \lesssim (\mathbf{C}_{\mu}\mathbf{C}_{\sigma}^{2}\|\boldsymbol{\mu}\|)\delta^{2s + t - 3 - \epsilon}, \end{align*} 
where we first used that $\sigma \sim \sigma_{p}(q) \lesssim \mathbf{C}_{\sigma}\delta^{s}$ by the $(\delta,s,\mathbf{C}_{\sigma})$-measure property of $\boldsymbol{\sigma}_{p}$, and finally $\mu|\mathcal{P}| \lesssim \|\boldsymbol{\mu}\|$ by property (a).

Finally, we move to the proof of \eqref{form34}. This is an easy application of Theorem \ref{t:fuRenParabolas} to the Katz-Tao $(\delta,t,O(\mathbf{C}_{\mu}\delta^{t}\mu^{-1}))$-set $\mathcal{P}$, and the Katz-Tao $(\delta,s,O(\mathbf{C}_{\sigma}\delta^{s}\sigma^{-1}))$-sets $\mathcal{F}(p) \cap  \mathcal{F}_{r}$. It is crucial that the Katz-Tao version of the $(\delta,s)$-set condition is inherited by subsets without altering the constants.

Now, recall that if $q \in  \mathcal{F}_{r}$, then by definition there exist $\geq r$ elements $p \in \mathcal{P}$ such that $q \in \mathcal{F}(p)$. Therefore,
\begin{displaymath} r|\mathcal{F}_{r}| \leq \sum_{q \in \mathcal{F}_{r}} |\{p \in \mathcal{P} : q \in \mathcal{F}(p)\}| = \sum_{p \in \mathcal{P}} |\mathcal{F}(p) \cap \mathcal{F}_{r}|\end{displaymath} 
Since all the sets $\mathcal{F}(p) \cap \mathcal{F}_{r}$ are contained in $\mathcal{F}_{r}$, an application of Theorem \ref{t:fuRenParabolas} gives the following upper bound for the right hand side:

\begin{displaymath} \sum_{p \in \mathcal{P}} |\mathcal{F}(p) \cap \mathcal{F}_{r}| \lesssim_{\epsilon} \delta^{-\epsilon} \sqrt{\delta^{-1} \cdot (\mathbf{C}_{\mu}\delta^{t}\mu^{-1}) \cdot (\mathbf{C}_{\sigma}\delta^{s}\sigma^{-1}) \cdot |\mathcal{F}_{r}||\mathcal{P}|}, \qquad \epsilon > 0. \end{displaymath} 
Combining and rearranging the previous two inequalities, 
\begin{displaymath} |\mathcal{F}_{r}| \lesssim_{\epsilon} \frac{\mathbf{C}_{\mu}\mathbf{C}_{\sigma}(\mu \sigma)^{-1}\delta^{s + t - 1 - \epsilon}|\mathcal{P}|}{r^{2}}, \qquad \epsilon > 0. \end{displaymath}
This is what was claimed in \eqref{form34}, so the proof is complete. \end{proof}

\begin{remark} For applications elsewhere, we finally record a variant of Proposition \ref{thm4} where the measures $\boldsymbol{\sigma}_{p}$ are supported near $\delta$-tubes instead of translates of the parabola $\mathbb{P}$. The proof is the same as the proof of Proposition \ref{thm4}, except that one applies the original Fu-Ren incidence estimate, Theorem \ref{t:fuRen}, instead of its "parabola version" Theorem \ref{t:fuRenParabolas}.

If $p \in \mathcal{D}_{\delta}$ with centre $(a,b)$, we write
\begin{displaymath} \ell_{p} = \{(x,y) : y = ax + b\} \quad \text{and} \quad T_{p} := [\ell_{p}]_{\delta/2}. \end{displaymath}
With this convention, the following holds: if $\mathcal{P} \subset \mathcal{D}_{\delta}$ is a Katz-Tao $(\delta,t,C)$-set of $\delta$-squares, the $\delta$-tube family $\{T_{p} : p \in \mathcal{P}\}$ is a Katz-Tao $(\delta,t,C')$-set of $\delta$-tubes with $C' \lesssim C$. \end{remark}

\begin{proposition}\label{thm4Tubes} Let $s \in [0,1]$ and $t \in (0,2)$ such that $s + t \leq 2$, and let $\mathbf{C}_{\mu},\mathbf{C}_{\sigma} \geq 1$. Let $\boldsymbol{\mu}$ be a $(\delta,t,\mathbf{C}_{\mu})$-measure. For each $p \in \mathcal{D}_{\delta}$, let $\boldsymbol{\sigma}_{p}$ be a $(\delta,s,\mathbf{C}_{\sigma})$-measure supported on $\mathcal{D}_{\delta}(T_{p}) = \{q \in \mathcal{D}_{\delta} : q \cap T_{p} \neq \emptyset\}$. Then,
\begin{displaymath} \int \Big( \sum_{p \in \mathcal{D}_{\delta}} \mu(p) \sum_{q \in \mathcal{D}_{\delta}} \sigma_{p}(q) \cdot (\delta^{-2}\mathbf{1}_{q}) \Big)^{2} \, dx \lesssim_{\epsilon} (\mathbf{C}_{\mu}\mathbf{C}_{\sigma}^{2}\|\boldsymbol{\mu}\|) \cdot \delta^{2s + t - 3 - \epsilon} + 1, \qquad \epsilon > 0. \end{displaymath} \end{proposition}

\bibliographystyle{plain}
\bibliography{references}

\end{document}